\pdfoutput=1

\documentclass{article}
\usepackage[utf8]{inputenc}
\usepackage{geometry} 
\usepackage{mathtools} 
\usepackage{amsmath}
\usepackage{color} 
\usepackage[dvipsnames]{xcolor}
\usepackage{setspace} 
\usepackage{amssymb} 
\usepackage{amsthm} 
\usepackage{titling}
\usepackage{graphicx}
\usepackage{comment} 
\usepackage{datetime} 
\usepackage{mathrsfs} 
\usepackage{subcaption}
\usepackage{morefloats}
\usepackage{pinlabel}
\usepackage{url}
\usepackage{pinlabel}
\usepackage{rotating}
\usepackage{sidecap}
\usepackage{enumerate}

\DeclarePairedDelimiter{\ceil}{\lceil}{\rceil}

\newtheorem{thm}{Theorem}[section]
\newtheorem{cor}[thm]{Corollary}
\newtheorem{lem}[thm]{Lemma}

\newtheorem{prop}[thm]{Proposition}
\newtheorem{obs}[thm]{Observation}

\title{Critical Bridge Spheres for Links with Arbitrarily Many Bridges}
\author{Puttipong Pongtanapaisan and Daniel Rodman}
\date{January 2019}

\begin{document}

\maketitle

\textbf{Abstract:} We show that for every integer \(b\geq 3\), there exists a link in a \(b\)-bridge position with respect to a critical bridge sphere.
In fact, for each \(b\), we construct an infinite family of links which we call \textit{square} whose bridge spheres are critical.

\section{Introduction}

For \(n\in\mathbb{N}\cup\{0\}\), a topological space is called \textit{\(n\)-connected} if the homotopy groups of dimension \(0,1,2,\dots,n\) are all trivial.
We call every topological space \((-2)\)-connected, and we call every nonempty topological space \((-1)\)-connected.

In \cite{bachman2010topological}, Bachman defined \textit{topologically minimal surfaces} as a topological analogue of geometrically minimal surfaces.
If \(\Sigma\) is a splitting surface in a compact, connected, orientable 3-manifold, the \textit{disk complex} \(\Gamma(\Sigma)\) is defined to be the simplicial complex whose vertices are isotopy classes of compressing disks for \(\Sigma\), and whose \(k\)-simplices are sets of \(k+1\) vertices with pairwise disjoint representatives.
Then \(\Sigma\) is defined to be \textit{topologically minimal of index \(n\)} if \(\Gamma(\Sigma)\) is \((n-2)\)-connected but not \((n-1)\)-connected. 
Thus by definition, the familiar categories of incompressible surfaces and strongly irreducible surfaces are equivalent to topologically minimal surfaces of index \(0\) and of index \(1\), respectively.
Topologically minimal surfaces can be characterized as those surfaces whose disk complex is empty or whose disk complex has at least one nontrivial homotopy group.
In \cite{bachman2010topological}, Bachman showed that topologically minimal surfaces of index \(2\) are precisely \textit{critical surfaces}, which were previously defined and utilized in \cite{bachman2008connected} to prove Gordon's conjecture.

Recall that a surface $F$ is said to be \textit{critical} if isotopy classes of compressing disks can be partitioned into two subsets $\mathcal{C}_0$ and $\mathcal{C}_1$ in such a way that: (1) There exists a pair of compressing disks $D_i,E_i \in \mathcal{C}_i$ on opposite sides of $F$ with the property that $D_i \cap E_i = \emptyset$, where $i \in \lbrace 0,1\rbrace$. (2) If $D \in \mathcal{C}_i$ and $E \in \mathcal{C}_{1-i}$ lie on opposite sides of $F$, then $D \cap E \neq \emptyset.$ For ease of visualization, we will refer to a
compressing disk $C$ as \textit{red} if $C \in \mathcal{C}_0$ and \textit{blue} if $C \in \mathcal{C}_1$.

Topologically minimal surfaces generalize important properties of incompressible surfaces despite being compressible (for index at least 1).
For instance, generalizing a well-known result of Haken \cite{haken1961theorie} that any incompressible surface can be isotoped to a normal surface, Bachman showed that a topologically minimal surface can be isotoped into a particular normal form with respect to a fixed triangulation \cite{bachman2012normalizing,bachman2012normalizing2,bachman2013normalizing3}.

Topological indices have been computed for various Heegaard surfaces of 3-manifolds \cite{bachman2010existence,campisi2018hyperbolic,lee2015topologically}, but less is known about the topogical indices of bridge surfaces. Lee has shown that the $(n + 1)$-bridge sphere for the unknot is a topologically minimal surface of index at most $n$ \cite{lee2016bridge}. This implies that the unknot in 3-bridge position has a corresponding bridge sphere whose index is at most 2, and in fact it is critical.
In \cite{rodman2018infinite}, we constructed a family of 4-bridge links and showed that their corresponding bridge spheres are critical, demonstrating the existence of critical bridge spheres for nontrivial multicomponent links.
This paper generalizes that result, showing that for every \(b\geq 3\), there is an infinite family of \(b\)-bridge links whose corresponding bridge spheres are critical.
In particular, we obtain the first known examples of critical bridge spheres for nontrivial knots.

This paper should be understood as a companion to \cite{rodman2018infinite}, of which this is a generalization.
Many of the steps we take here are straightforward generalizations of steps taken there, and for the sake of brevity, we have here omitted several of the details and proofs that are essentially identical to those found in \cite{rodman2018infinite}.

\begin{figure}[h!]
    \centering
    \includegraphics[width=.5\textwidth]{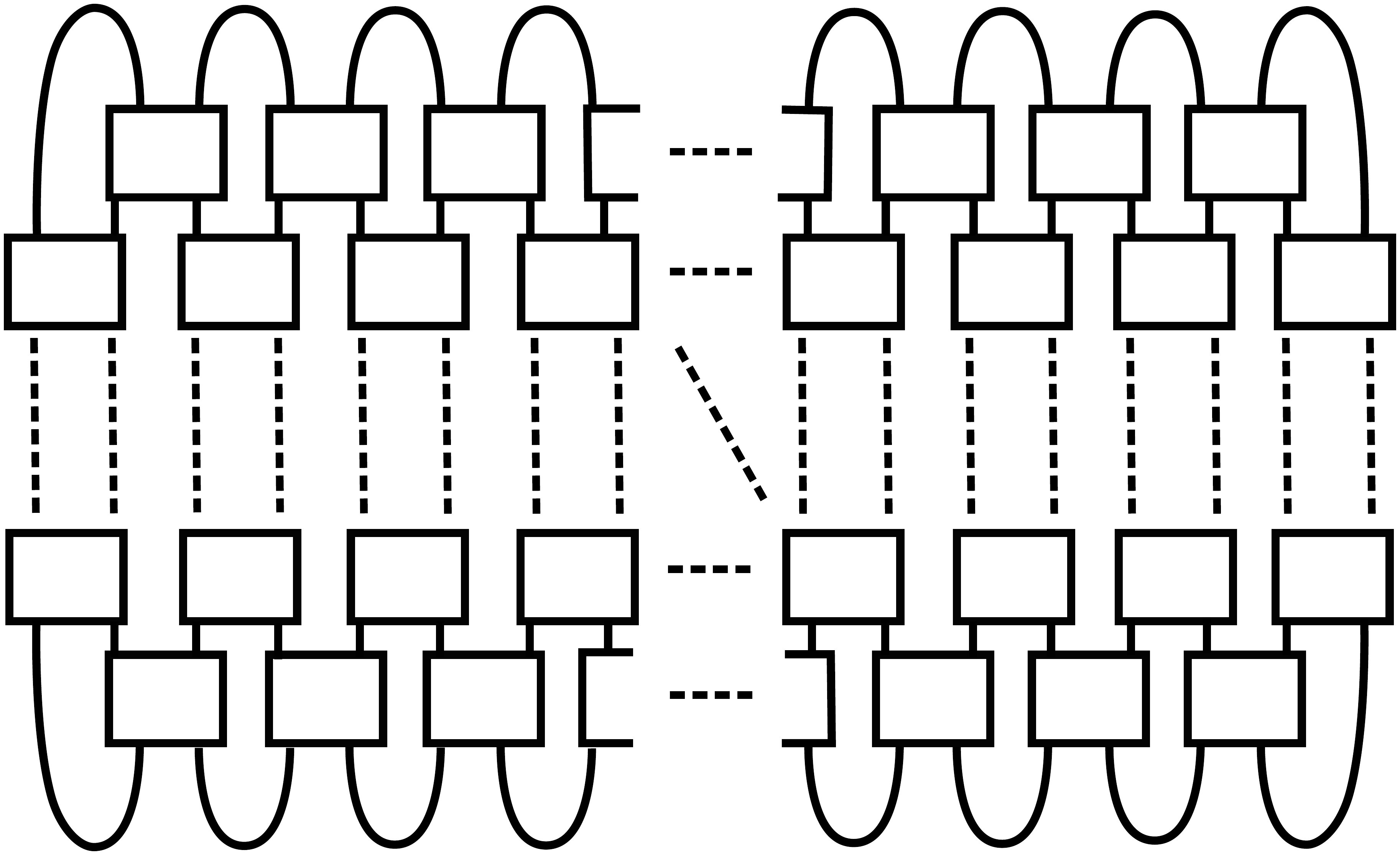}
    \caption{A link in a plat position. Each box represents a twist region containing some number of half twists.}
    \label{fig:plat_link_general}
\end{figure}

\begin{figure}[h!]
    \centering
    \includegraphics[width=.25\textwidth]{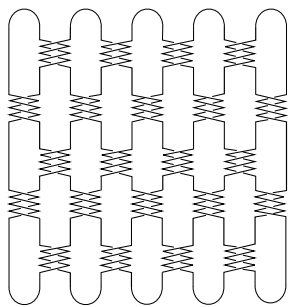}
    \caption{An example of a square  plat position for a link.
    This link is in \({(h,b)}\)-plat position, where \(h=6\) is the height (one more than the number of rows of twist regions), and \(b=5\) is the number of bridges.
    The position is \textit{square} since \(h=2b-4\).
    This example is also \(2\)-twisted because each twist region contains at least \(2\) half twists.}
    \label{fig:plat_link_example}
\end{figure}

Every link \(L\) has a plat position, which is an embedding into \(S^3\) such that \(L\) is a union of vertical strands, twist regions, and bridges, arranged as in Figure \ref{fig:plat_link_general}.
We will call such an embedding a \textit{plat link}. 
The number \(b\) of bridges may be any integer greater than zero.
The number of rows of twist regions may be any integer greater than zero as well, and we define the \textit{height} of a plat link to be one more than the number of rows of twist regions.
A plat link with \(b\) bridges and height \(h\) is called an \textit{\({(h,b)}\)-plat link}.
For example, the plat link in Figure \ref{fig:plat_link_example} is a \({(6,5)}\)-plat link.
Plat links are more carefully defined in \cite{johnson2016bridge}.

Each twist region in a plat link may contain any integral number of half twists, all twisting in the same direction.
A plat link is called \textit{\(n\)-twisted} if every twist region contains at least \(n\) half twists (in either direction).

The \textit{bridge distance} of a splitting surface \(\Sigma\) is the path length of a shortest path in \(\Sigma\)'s curve complex whose endpoints are represented by two disjoint loops which are boundaries of compressing disks on opposite sides of \(\Sigma\).
In Theorem 1.2 from \cite{johnson2016bridge}, Johnson and Moriah show that the bridge distance of the bridge sphere corresponding to an \({(h,b)}\)-plat link is equal to \(\ceil{h/(2b-4)}\).
Thus for a fixed number of bridges, the greater the value of \(h\) (i.e., the ``taller" the plat link), the greater the bridge distance, and similarly, for a fixed height \(h\), the quantity \(b\) can be increased to decrease the bridge distance.
Informally, ``tall" plat links have high bridge distance, and sufficiently ``wide" links have bridge distance one.

It can be shown that every high distance bridge surface has topological index \(1\). Thus, to obtain bridge surfaces with slightly higher topological index (i.e., index \(2\)), we consider links that are found on the boundary of ``wide" and ``tall," and we call them ``square."
We define a \textit{square} plat link to be an \({(h,b)}\)-plat link with \(h=2b-4\).
Notice that in this case, according to Johnson and Moriah's theorem, this implies that the bridge distance is \(1\).
These links are just wide enough to avoid being strongly irreducible (and thus index \(\leq 1\)), but still tall enough to have most pairs of compressing disks on opposite sides of the bridge sphere intersect each other. This paper shows that under a certain assumption on the twist regions, the canonical bridge sphere of a square plat link has topological index two. This result can be viewed as a step towards a way to determine the topological index from a plat projection completely in terms of the height and the number of bridges.

\section{Setting}\label{sec:setting} 
Johnson and Moriah carefully define plat links in \cite{johnson2016bridge}.
They identify \(\mathbb{R}^3\) with an open ball in \(S^3\) and embed their plat link inside that ball, giving them the convenience of working within a Cartesian coordinate system.
This allows them to use concepts like ``left," ``right," and ``straight."

\begin{figure}
     \centering
\labellist \small\hair 2pt
\pinlabel {Level \(1\)} [l] at -111 50
\pinlabel {Level \(h-2\)} [l] at -111 128
\pinlabel {Level \(h-1\)} [l] at -111 215
\pinlabel {Level \(h\)} [l] at -111 305
\pinlabel {\(B\)} at 25 343
\pinlabel {\(\alpha^1\)} at 87 344
\pinlabel {\(\alpha^2\)} at 183 344
\pinlabel {\(\alpha^3\)} at 279 344
\pinlabel {\(\alpha^4\)} at 375 344
\pinlabel {\textcolor{white}{\(D^1\)}} at 87 322
\pinlabel {\textcolor{white}{\(D^2\)}} at 183 322
\pinlabel {\textcolor{white}{\(D^3\)}} at 279 322
\pinlabel {\textcolor{white}{\(D^4\)}} at 375 322
\pinlabel {\(\beta^1\)} at 38 289
\pinlabel {\(\beta^2\)} at 185 298
\pinlabel {\(\beta^3\)} at 282 298
\pinlabel {\(\beta^4\)} at 379 298
\pinlabel {\(\gamma^1\)} at 136 298
\pinlabel {\(\gamma^2\)} at 233 298
\pinlabel {\(\gamma^3\)} at 330 298
\pinlabel {\(\gamma^4\)} at 423 298
\endlabellist
    \includegraphics[width=.75\textwidth]{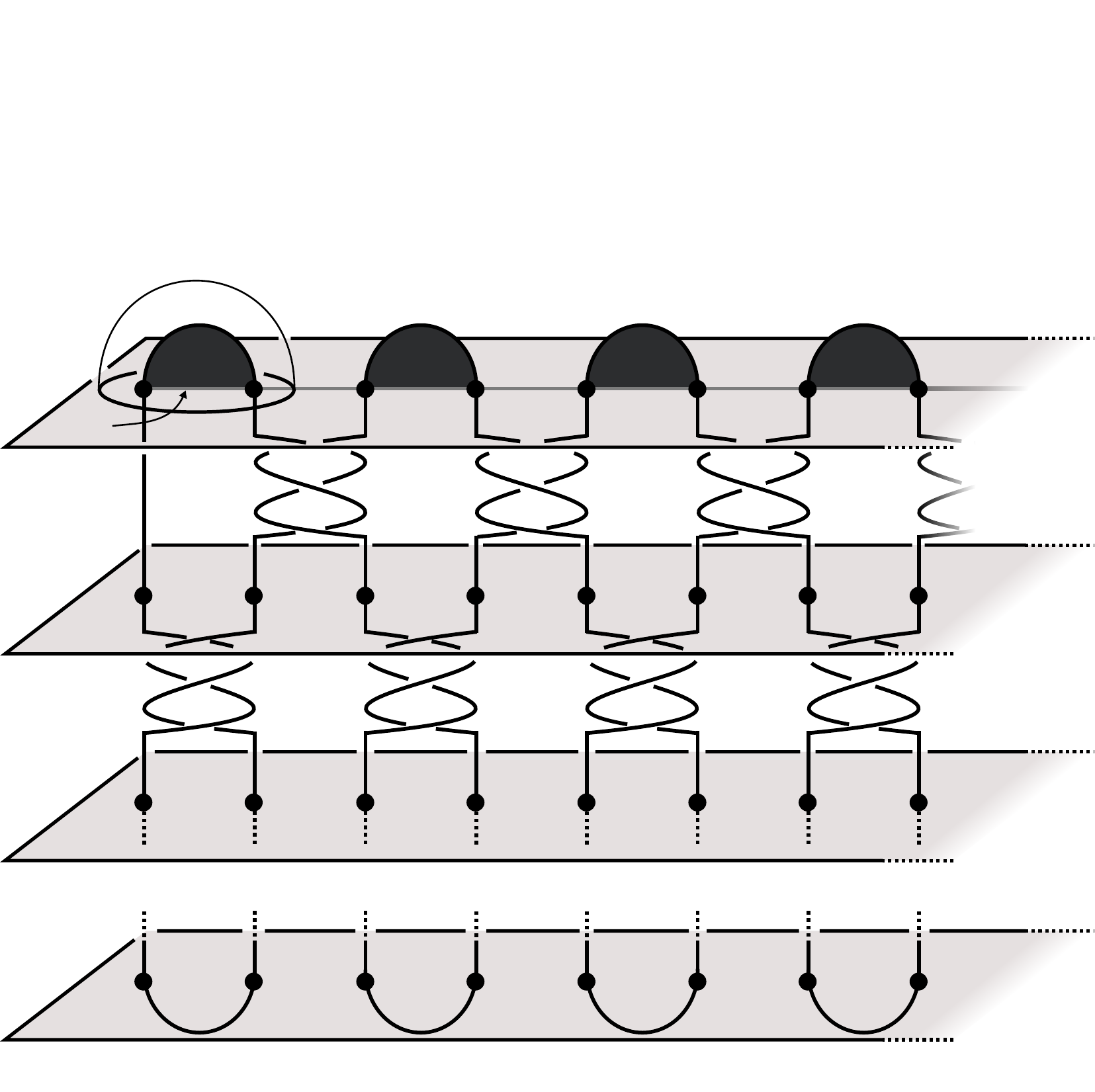}
    \caption{The bridge sphere \(F\) is depicted at various levels.
    The dark disks are the bridge disks above $F$.
    The compressing disk \(B\) is pictured at the top of the figure.}
    \label{fig:190417_PlatLink_Details}
\end{figure}
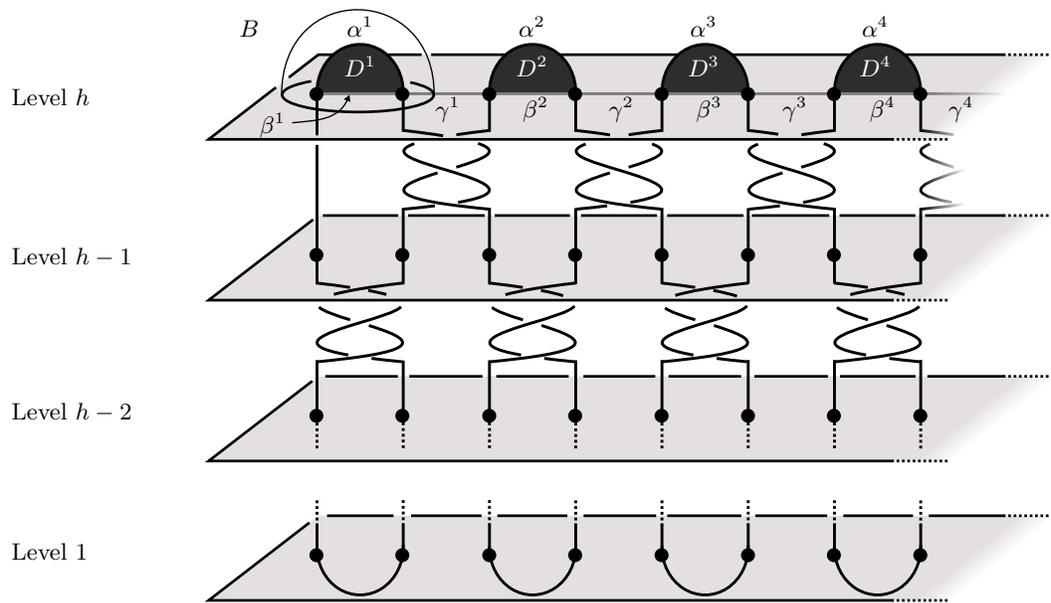

We will follow Johnson and Moriah's example here. 
Let \(L\) be a \(2\)-twisted square plat link (so \(L\) is in \({(h,b)}\)-plat position with \(h=2b-4\)).
Further, assume for the rest of the paper that the twist regions all twist positively in the bottom row of twist regions, negatively along the next row up, and so on, alternating from positive to negative as we move up from row to row.
Let $S$ be the canonical bridge sphere for \(L\).
\(S\) decomposes $S^3$ into two 3-balls $M_+$ and $M_-$.
Define \(F = S \backslash L \), which can be locally visualized as a punctured horizontal plane, and it is convenient to consider \(F\) embedded at different levels of \(L\). 
Specifically, \(F\) may be embedded just below the \(n\)th row of twist regions, which we call \textit{level \(n\)}.
When embedded just below the top row of bridges, we say \(F\) is at \textit{level \(h\)}.
See Figure \ref{fig:190417_PlatLink_Details}.

Consider \(F\) embedded at level \(h\).
We will name the upper bridge arcs \(\alpha^1,\alpha^2,\dots,\alpha^b\), from left to right, respectively.
The upper bridge arcs intersect \(F\) in  \(2b\) punctures, all in a straight line.
Let \(\Gamma\) be the straight line segment containing all these punctures.
\(\Gamma\) is cut into \(2b-1\) subarcs by the punctures of \(L\cap F\). 
For \(1\leq i\leq b\), we define \(\beta^i\) to be  the subarc of \(\Gamma\) whose endpoints are the endpoints of the bridge \(\alpha^i\).
For \(1\leq i\leq b-1\), we define \(\gamma^i\) to be the arc of \(\Gamma\) which shares its left endpoint with \(\alpha^i\) and its right endpoint with \(\alpha^{i+1}\).
Thus \[\Gamma=\beta^1\cup\gamma^1\cup\beta^2\cup\gamma^2\cup\cdots\cup\gamma^{i-1}\cup\beta^b.\]
For each \(i\), the pair of arcs \(\alpha^i\), \(\beta^i\) bounds a flat disk \(D^i\) pictured in Figure \ref{fig:190417_PlatLink_Details}.
Collectively, we will refer to the disks \(D^i\) as the \textit{(upper) bridge disks}.

We will also label some of the corresponding elements below the bridge sphere.
Consider \(F\) embedded at level 1.
We will name the rightmost (resp. leftmost) lower bridge arc \(\alpha'\) (resp. \(\alpha''\)), and we will define \(\beta'\) (resp. \(\beta''\)) to be the straight line segment in \(F\) between the endpoints of \(\alpha'\) (resp. \(\alpha''\)). 
Observe that together, \(\alpha'\) and \(\beta'\) bound a lower bridge disk which we call \(D'\). Similarly, \(\alpha''\) and \(\beta''\) cobound a bridge disk \(D''\).

For the rest of the paper, \(B\) will denote the compressing disk obtained from taking the frontier of a regular neighborhood of \(D^1\) in $M_+$ (see Figure \ref{fig:190417_PlatLink_Details}).
Similarly, \(B'\) will denote the compressing disk obtained from taking the frontier of a regular neighborhood of \(D'\) in $M_-$.
We are going to partition the set of compressing disks as follows.
The set of blue compressing disks will consist of exactly two disks: \(B\) and \(B'\), and all other compressing disks for \(F\) will be red. We will show that our choice for the partition satisfies the conditions in the definition of a critical surface.

\section{The labyrinth}

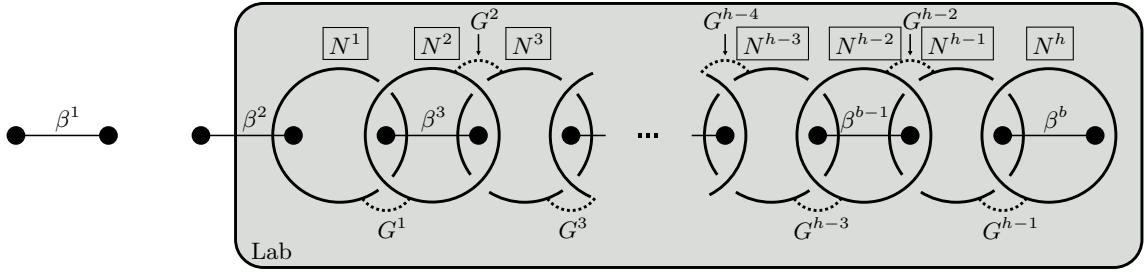
\begin{figure}
    \centering
\labellist \small\hair 2pt
\pinlabel {
    \setlength{\fboxsep}{2pt}
	\fbox{\(N^1\)}
	} at 165 110
\pinlabel {
    \setlength{\fboxsep}{2pt}
	\fbox{\(N^2\)}
    } at 209.5 110
\pinlabel {
    \setlength{\fboxsep}{2pt}
	\fbox{\(N^3\)}
    } at 254 110
\pinlabel {
    \setlength{\fboxsep}{2pt}
	\fbox{\(N^{h-3}\)}
    } at 373 110
\pinlabel {
    \setlength{\fboxsep}{2pt}
	\fbox{\(N^{h-2}\)}
    } at 418 110
\pinlabel {
    \setlength{\fboxsep}{2pt}
	\fbox{\(N^{h-1}\)}
    } at 463 110
\pinlabel {
    \setlength{\fboxsep}{2pt}
	\fbox{\(N^h\)}
    } at 508 110
\pinlabel {\(G^2\)} at 233 122
\pinlabel {\(G^{h-4}\)} at 353 122
\pinlabel {\(G^{h-2}\)} at 450 122
\pinlabel {\(G^1\)} at 188 20
\pinlabel {\(G^3\)} at 276 20
\pinlabel {\(G^{h-3}\)} at 397 20
\pinlabel {\(G^{h-1}\)} at 489 20
\pinlabel {\(\beta^1\)} at 30 73
\pinlabel {\(\beta^2\)} at 121 73
\pinlabel {\(\beta^3\)} at 208 73
\pinlabel {\(\beta^{b-1}\)} at 418 73
\pinlabel {\(\beta^{b}\)} at 511 73
\pinlabel {Lab} at 129 9
\endlabellist
\includegraphics[width=1\textwidth]{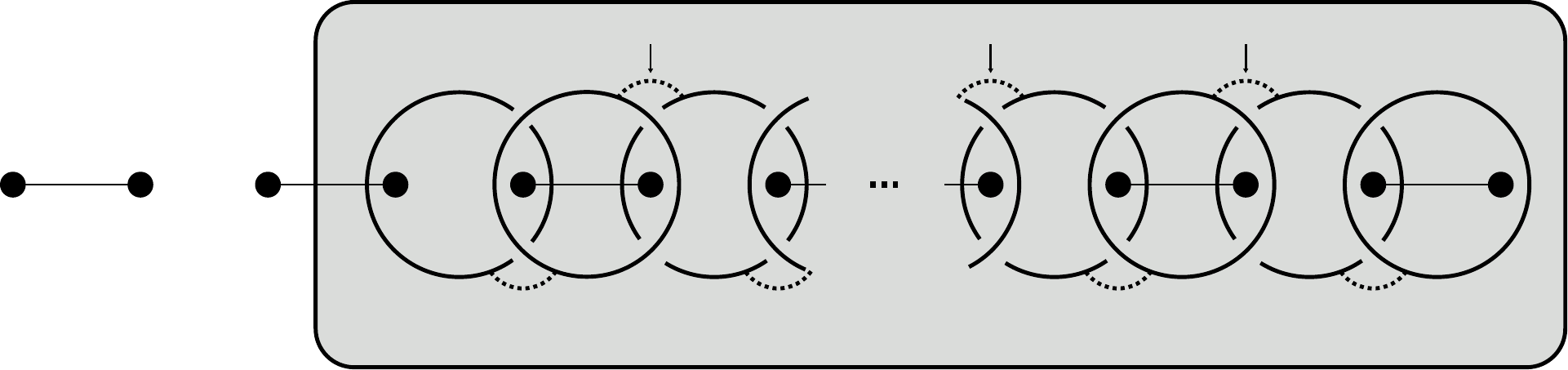}
\caption{An unlink diagram representing the image of \(\partial B'\) in \(F\) at level \(h\).
For each \(j\), \(N^j\) is the number of parallel circles represented by the corresponding circle in the diagram.}
\label{fig:190131GeneralUnlinkDiagram}
\end{figure}

\begin{figure}[h!]
\centering
\labellist \small\hair 2pt
\pinlabel {
	\setlength{\fboxsep}{2pt}
	\fbox{\(N\)}
	} at 20 56
\pinlabel {
	\setlength{\fboxsep}{2pt}
	\fbox{\(n\)}
	} at 20 30.5
\pinlabel {
	\setlength{\fboxsep}{2pt}
	\fbox{\(N\)}
	} at 20 134
\pinlabel {
	\setlength{\fboxsep}{2pt}
	\fbox{\(n\)}
	} at 20 160
\pinlabel {=} at 101 44
\pinlabel {=} at 101 145
\pinlabel {=} at 213 44
\pinlabel {=} at 213 145
\pinlabel {
\begin{rotate}{45}
\(
\left\{
\begin{tabular}{c}
~ \\ 
~ \\ 
\end{tabular}
\right.
\)
\end{rotate}
} at 122 105
\pinlabel {\(N\)} at 114 103
\pinlabel {
\begin{rotate}{225}
\(
\left\{
\begin{tabular}{c}
~ \\ 
~ \\ 
\end{tabular}
\right.
\)
\end{rotate}
} at 194 183
\pinlabel {\(N\)} at 200 185
\pinlabel {
\begin{rotate}{135}
\(
\left\{
\begin{tabular}{c}
~ \\ 
~ \\ 
\end{tabular}
\right.
\)
\end{rotate}
} at 197 9
\pinlabel {\(N\)} at 200 5
\pinlabel {
\begin{rotate}{-45}
\(
\left\{
\begin{tabular}{c}
~ \\ 
~ \\ 
\end{tabular}
\right.
\)
\end{rotate}
} at 121 80
\pinlabel {\(N\)} at 118 85
\pinlabel {
\begin{rotate}{-45}
\(
\left\{
\begin{tabular}{c}
~ \\ 
\end{tabular}
\right.
\)
\end{rotate}
} at 120 181
\pinlabel {\(n\)} at 113 187
\pinlabel {
\begin{rotate}{135}
\(
\left\{
\begin{tabular}{c}
~ \\ 
\end{tabular}
\right.
\)
\end{rotate}
} at 196 110
\pinlabel {\(n\)} at 197 104
\pinlabel {
\begin{rotate}{225}
\(
\left\{
\begin{tabular}{c}
~ \\ 
\end{tabular}
\right.
\)
\end{rotate}
} at 195 82
\pinlabel {\(n\)} at 199 82
\pinlabel {
\begin{rotate}{45}
\(
\left\{
\begin{tabular}{c}
~ \\ 
\end{tabular}
\right.
\)
\end{rotate}
} at 123 7
\pinlabel {\(n\)} at 118 5
\endlabellist
\includegraphics[scale=.9]{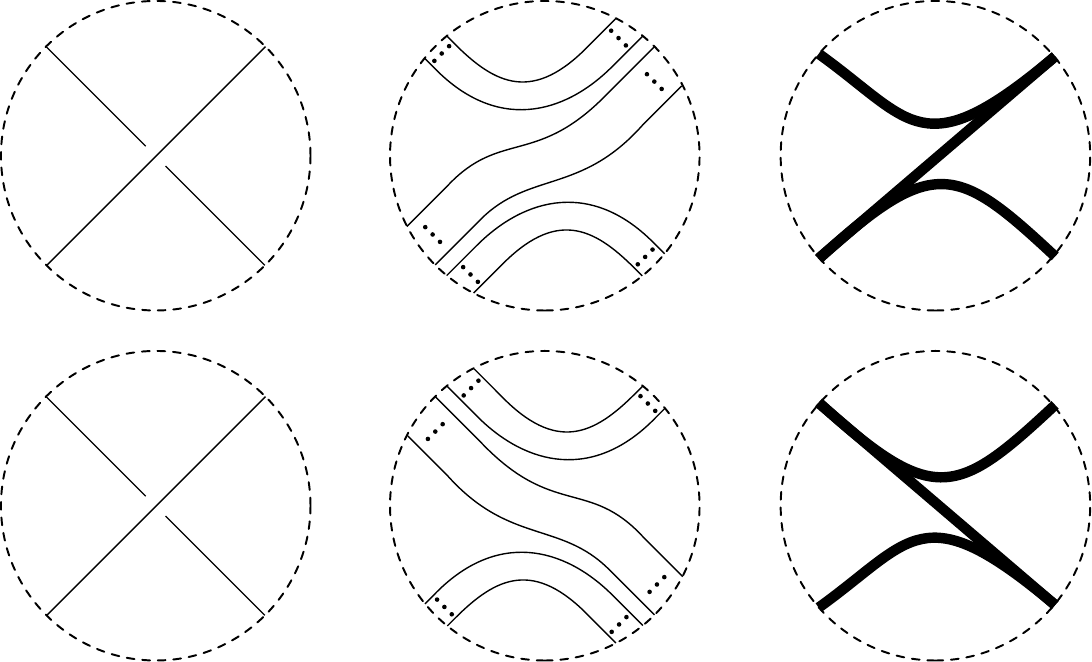}
\caption{
Here we depict how to interpret a crossing in the link diagram of \(N\) strands with \(n\) strands, where \(N>n\).
In the left pictures, we see the crossing in the unlink diagram.
In the middle pictures, we see the actual arcs that the crossing represents.
In the right pictures, we show how this unlink diagram relates to train tracks: each crossing in our unlink diagram is equivalent to two switches in a train track diagram.
}
\label{fig:190131multi_strand_crossing_general}
\end{figure}
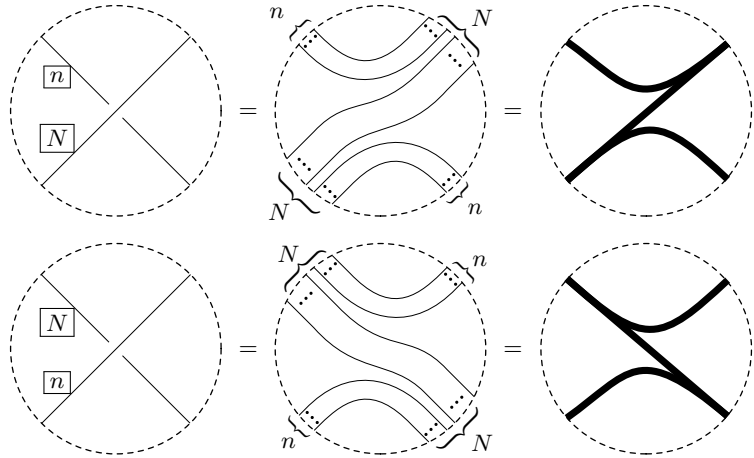

Under the isotopy of \(F\) taking \(F\) from level \(1\) to level \(h\), the loop \(\partial B'\) becomes incredibly convoluted.
In \cite{rodman2018infinite}, the author describes a convenient way to represent the image of \(\partial B'\) under this isotopy using the unlink diagram in Figure \ref{fig:190131GeneralUnlinkDiagram}.
To obtain Figure \ref{fig:190131GeneralUnlinkDiagram}, we use the fact that the signs of the rows of twist regions alternate, but we refer readers to \cite{rodman2018infinite} for more detail.
In this unlink diagram, each circle represents a number \(N^j\) of parallel copies of that circle.
(In figures, we will use boxed numbers \fbox{\(N\)} when denoting the number of parallel strands represented.)
Each crossing in the unlink diagram represents the configuration of arcs in Figure \ref{fig:190131multi_strand_crossing_general}.

Let \(G^1,G^2,\dots,G^{h-1}\) denote the dotted arcs in Figure \ref{fig:190131GeneralUnlinkDiagram}.
(For each \(i\), \(G^i\cap\partial B'=\partial G^i\).)
We will refer to these arcs as the \textit{gates}, and we will consider \(G^i\) to be ``colored" with color \(i\).
Notice that none of the gates is parallel to \(\partial B'\).

\begin{lem}
\(N^1>N^2\), and \(N^j>N^{j\pm 1}\) for odd \(j\geq 3\).
\end{lem}

\begin{proof}
Notice that in the unlink diagram in Figure \ref{fig:190131multi_strand_crossing_general}, every circle component has either all under-crossings or all over-crossings.
We will thus refer to the circles as \textit{undercircles} or \textit{overcircles}. 
With this language, this lemma can be stated, \textit{Any particular undercircle represents strictly more parallel circles than its adjacent overcircle(s)}.

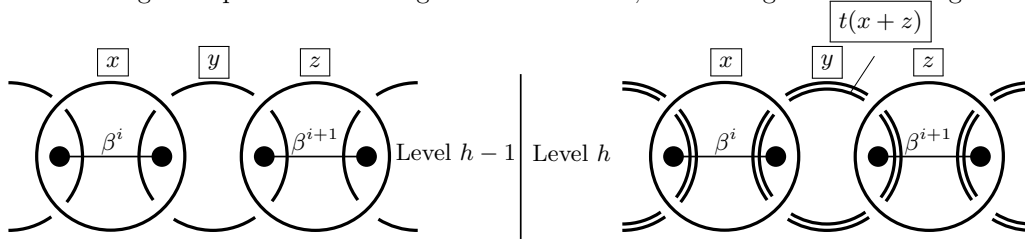
\begin{figure}
\centering
\labellist \small\hair 2pt
\pinlabel {Level \(h-1\)} at 197 38
\pinlabel {Level \(h\)} at 248 38

\pinlabel {\fbox{\(x\)}} at 46 78
\pinlabel {\fbox{\(y\)}} at 90.5 78
\pinlabel {\fbox{\(z\)}} at 135 78
\pinlabel {\fbox{\(x\)}} at 316 78
\pinlabel {\fbox{\(y\)}} at 360.5 78
\pinlabel {\fbox{\(t(x+z)\)}} at 384 96
\pinlabel {\fbox{\(z\)}} at 405 78
\pinlabel {\(\beta^i\)} at 46 43
\pinlabel {\(\beta^{i+1}\)} at 135 43
\pinlabel {\(\beta^i\)} at 316 43
\pinlabel {\(\beta^{i+1}\)} at 405 43
\endlabellist
\includegraphics[width=.9\textwidth]{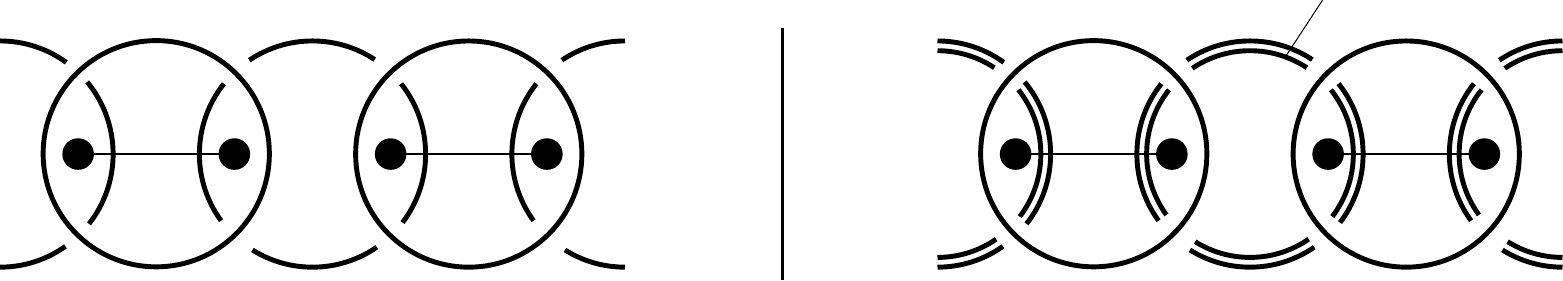}
\caption{On the left, we see the loop diagram for \(\partial B'\) at level \(h-1\), where the three full loops represent \(x\), \(y\), and \(z\) strands, respectively from left to right.
In the figure on the right, the unlink diagram has been pushed up to level \(h\), an isotopy during which the second and third punctures twist through a twist region, resulting in the addition of \(t(x+z)\) new loops (where \(t\) is the number of half twists in the twist region).
Thus \(N^{j-1}=x\), \(N^j=y+t(x+z)\), and \(N^{j+1}=z\), and so \(N^j>N^{j\pm 1}\).
}
\label{fig:190117OverAndUnderCircles2}
\end{figure}

Consider the loop \(\partial B'\) embedded in \(F\) at level 1.
As we ``push \(F\) up" to level \(h\), \(\partial B'\) is forced to become considerably twisted as it spirals through the twist regions of \(L\).
To create the unlink diagram, we add a number of overcircles or undercircles to the unlink diagram for each twist region that \(F\) passes through.
(This process is thoroughly explained and illustrated in \cite{rodman2018infinite}.)
Refer to Figure \ref{fig:190117OverAndUnderCircles2}, which depicts, locally, the situation when we push \(F\) up from level \(h-1\) to level \(h\).
The left half of the figure depicts the link diagram at level \(h-1\), and the right half of the figure depicts the link diagram at level \(h\).
On each side of the figure, the first two punctures and the last two punctures lie below bridge arcs.
On the left, the figure depicts an undercircle and the two overcircles on either side of it.
The circles represent \(x\), \(y\), and \(z\) parallel circles, from left to right (where \(x\geq 0\) and \(y,z\geq 1\)).
When we push \(F\) up to level \(h\), the second and third punctures in the figure pass through a twist region, and so a new circle is added to the diagram around them.
To determine how many parallel strands this new circle represents, the rule is that the number of parallel strands of \(\partial B'\) which intersect the newly added circle at level \(h-1\) (here, \((x+z)\) strands) is multiplied by the absolute value of the number \(t\) of half twists in the twist region (in the top row, the twist numbers are all positive, so \(|t|=t\)).
Now there are two parallel circles pictured, the original one representing a total of \(y\) parallel circles, and the newly added one representing \(t(x+z)\) parallel circles.
All of those circles are parallel, and so in this picture, after isotopy of \(\partial B'\) to level \(h\), the total number of parallel circles represented by \(N^j\) will be \(N^j=t(x+z)+y\).
Since \(j\) is odd, \(N^j\) may be \(N^1\), in which case \(x=0\), but in any case, \(z\) is nonzero, and \(t\geq 2\) since \(L\) is 2-twisted.
Therefore whether or not \(x=0\), we have by basic algebra that \(t(x+z)+y>x\) and \(t(x+z)+y>z\), and so we conclude that \(N^j>N^{j\pm 1}\).
\end{proof}

\begin{cor}\label{cor:lab_gate_locations}
In the unlink diagram representing \(\partial B'\) (Figure \ref{fig:190131GeneralUnlinkDiagram}), every crossing represents the bottom picture of Figure \ref{fig:190131multi_strand_crossing_general} rather than the top picture.
\end{cor}

\begin{prop}\label{prop:min_pos}
The position of \(\partial B'\) described by the link diagram in Figure \ref{fig:190131GeneralUnlinkDiagram} is minimal with respect to \(\beta^2\cup\beta^3\cup\cdots\cup\beta^b\). 
\end{prop}

\begin{proof}
The proof of Proposition \ref{prop:min_pos} is similar to the proof of Proposition 4.1 in \cite{rodman2018infinite}.
The latter proves the statement for the specific case when \(b=4\), and it generalizes directly to arbitrary \(b\). 
\end{proof}

Let \textbf{Lab} be the disk in \(F\) at level \(h\) which contains \(\partial B'\) and whose complement in \(F\) is a regular neighborhood of \(\beta^1\cup\gamma^1\).
This is the gray disk pictured in Figure \ref{fig:190131GeneralUnlinkDiagram}.

\begin{figure}
    \centering
    \includegraphics[width=.2\textwidth]{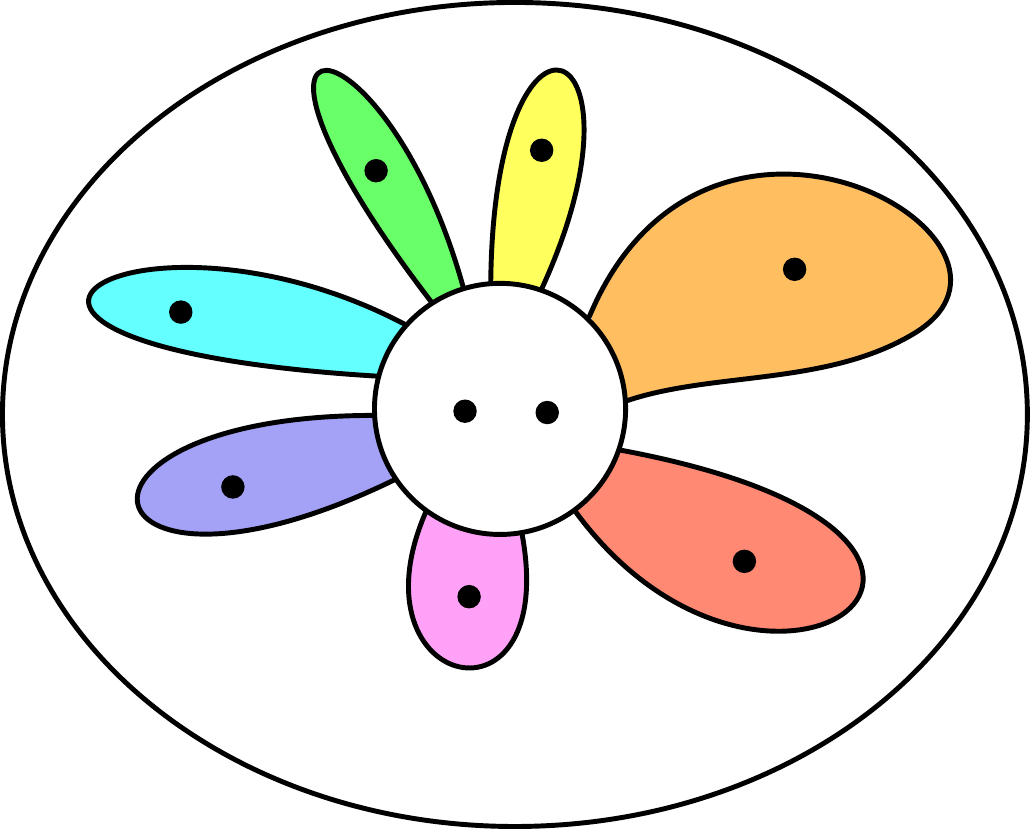}
    \caption{The disk Lab is isotopic to this picture. 
    The dotted lines are the gates, of which there are \(h-1\).
    The center circle is \(\partial B'\).
    The punctured disks cut out by the gates and by \(\partial B'\) are the colored punctured disks.}
    \label{fig:190119Flower_colored}
\end{figure}

\begin{prop}\label{prop:flower}
\(\partial B'\) and the gates cut Lab into \(h+1\) components: \(h-1\) once-punctured disks, an annulus, and a twice-punctured disk.
\end{prop}

\begin{proof}
The proof is similar to the proof of Proposition 4.2 in \cite{rodman2018infinite}, where the statement is proved for the specific case where \(b=4\), but it generalizes simply by replacing every instance of the number 3 with \(b-1\).
\end{proof}

Consider the once-punctured disks cut out of Lab by \(\partial B'\) and the gates (i.e., the ``petals" in Figure \ref{fig:190119Flower_colored}).
Each one is bounded by one gate and by an arc of \(\partial B'\).
For each color \(i\), we will refer to the once-punctured disk whose boundary contains \(G^i\) as the \textbf{\(i\)-colored disk}, and we will refer to these punctured disks collectively as the \textbf{colored punctured disks}.

\section{Red Disks Enter the Labyrinth}

\begin{figure}[ht!]
\labellist
\pinlabel $F$ at 60 34
\pinlabel $\alpha^{j+1}$ [l] at 380 65
\pinlabel $\alpha^j$  at 253 50
\pinlabel $R$ [r] at 150 150
\endlabellist
\centering
\includegraphics[scale=0.65]{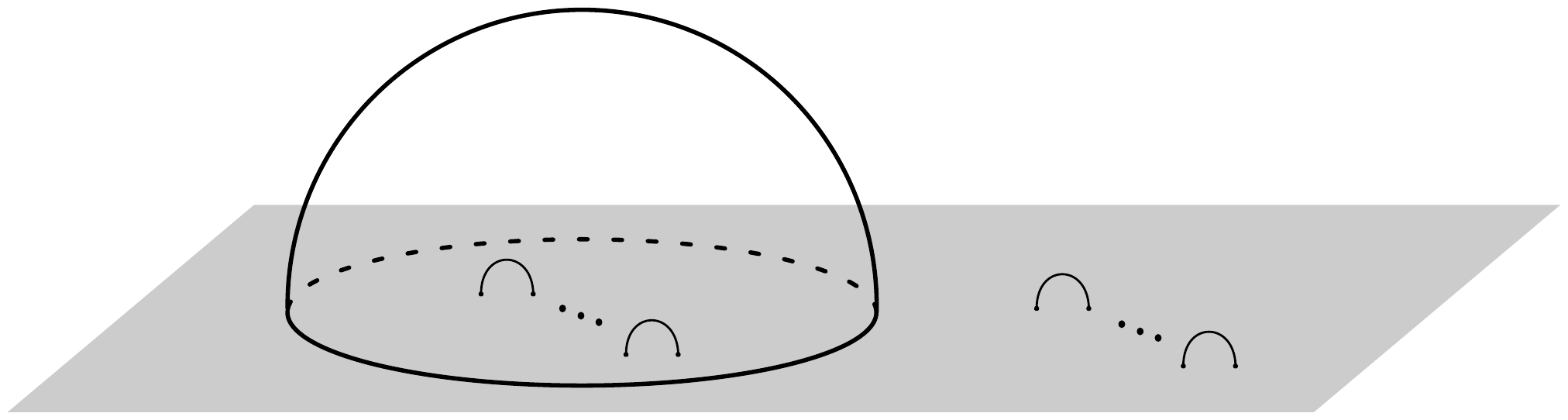} 
\caption{The red disk $R$ separates $\alpha^j$ from $\alpha^{j+1}$ for $2 \leq j \leq b-1$}
\label{fig:case1lem4.1}
\end{figure}
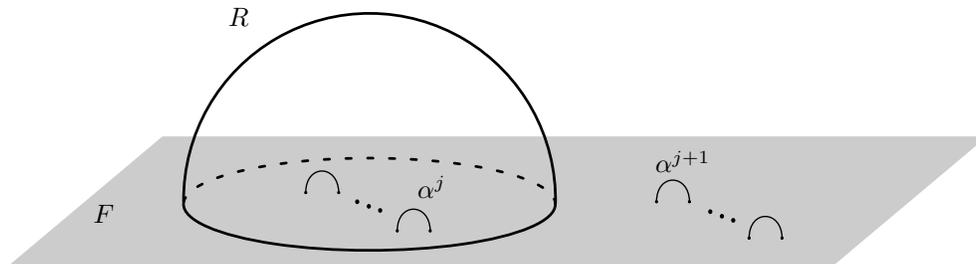

\begin{figure}[ht!]
\labellist
\pinlabel $F$ at 60 34
\pinlabel $\alpha^{1}$ [l] at 360 65
\pinlabel $\widetilde{D}$  at 213 100
\pinlabel $R$ [r] at 150 150
\endlabellist
\centering
\includegraphics[scale=0.65]{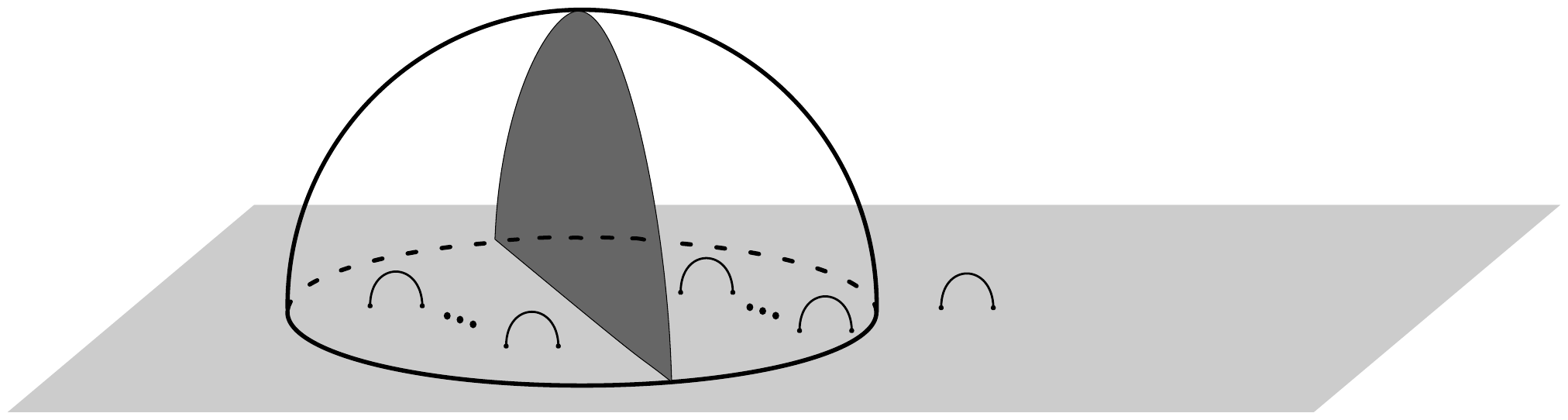} 
\caption{$\widetilde{D}$ lies in $M_+^1$ and divides $M_+^1$ into two components.}
\label{fig:case2lem4.1}
\end{figure}

\begin{lem}\label{lem:intersect_gamma_arc}
 If $R$ is a red disk above $F$, then $R \cap \left(\displaystyle\bigsqcup_{i=2}^{b-1} \gamma^i \right) \neq \emptyset.$
\end{lem}
\begin{proof}
The red disk $R$ separates $M_+$ into two components $M_+^1 \cup M_+^2$. We consider two cases:

\textbf{Case 1:} Suppose that $\alpha^2, \cdots ,\alpha^b$ do not lie in the same component. 
This means that for some index $2\leq j \leq b-1$, the arcs $\alpha^j$ and $\alpha^{j+1}$ lie in different components as in Figure \ref{fig:case1lem4.1}. 
Since $\gamma^j$ connects an endpoint of $\alpha^{j}$ to an endpoint of $\alpha^{j+1}$ in $F$, it holds that $R \cap \gamma^j \neq \emptyset.$

\textbf{Case 2:} Suppose that $\alpha^2,\cdots,\alpha^b$ lie in the same component $M_+^1$ of $M_+ \backslash R$, as in Figure \ref{fig:case2lem4.1}.
Observe that for $\partial R$ to be nontrivial in $F$, the disk $R$ must separate $\alpha^1$ from $\alpha^2,\cdots,\alpha^b$.
Let $\mathcal{D}$ denote the union of bridge disks \(D^2,D^3\cdots,D^b\).
If $R$ can be made disjoint from $\mathcal{D},$ either \(\partial R\) intersects \(\gamma^j\) for some \(j\geq 2\) (in which case, we are done), or $\partial R$ separates $F$ into a disk containing $\partial \beta^1$ and a disk containing $\xi = \beta^2 \cup \gamma^2 \cup \cdots \cup \gamma^{b-1} \cup \beta^b$.
This implies that $R$ is isotopic to the blue disk above $F$, which is a contradiction.
Therefore we may assume that $R$ intersects $\mathcal{D}$ transversely and minimally so that $R \cap \mathcal{D}$ consists of a positive number of arc components. 

Thus there exists some \(k\) such that $R \cap D^k \neq \emptyset.$
Then, there is an arc of intersection $\widetilde{\alpha}$ on $D^k$ and a subarc $\widetilde{\beta}$ of $\beta^k$ such that $\widetilde{\alpha} \cup \widetilde{\beta}$ cobound an outermost disk $\widetilde{D} \subset D^k$.
See Figure \ref{fig:case1lem5.1}.
By minimality of $|R \cap D^k|$, the disk $\widetilde{D}$ cannot lie in $M_+^2$, and so $\widetilde{D}$ lies in $M_+^1$ and divides $M_+^1$ into two components.
If $\alpha^2,\cdots,\alpha^b$ all lie in the same component of $M_+^1\backslash \widetilde{D}$, then the other side of $\widetilde{D}$ is a 3-ball that gives rise to an isotopy reducing $|R \cap \widetilde{D}|$, a contradiction.
This means that for some index $2\leq j \leq b-1$, the arcs $\alpha^j$ and $\alpha^{j+1}$ lie in different components of $M_+^1\backslash \widetilde{D}$.
The arc $\gamma^j$ does not intersect any bridge disks, so in particular,  $\gamma^j$ does not intersect \(\widetilde{D}\).
Therefore \(\gamma^j\) must intersect $\partial R$ nontrivially in order to connect $\alpha^j$ with $\alpha^{j+1}$.
Thus, $R \cap\left(\displaystyle\bigsqcup_{i=2}^{b-1} \gamma^i \right)\neq \emptyset$ as desired.
\end{proof}

\begin{figure}[ht!]
\labellist
\pinlabel $\widetilde{\beta}$ at 285 18
\pinlabel $F$ at 50 30
\pinlabel $\widetilde{D}$  at 287 55
\pinlabel $D^k$ at 293 150
\pinlabel $\gamma^k$ at 430 55
\pinlabel $\gamma^{k-1}$ at 150 55
\pinlabel $\alpha^k$ at 220 155
\endlabellist
\centering
\includegraphics[scale=0.65]{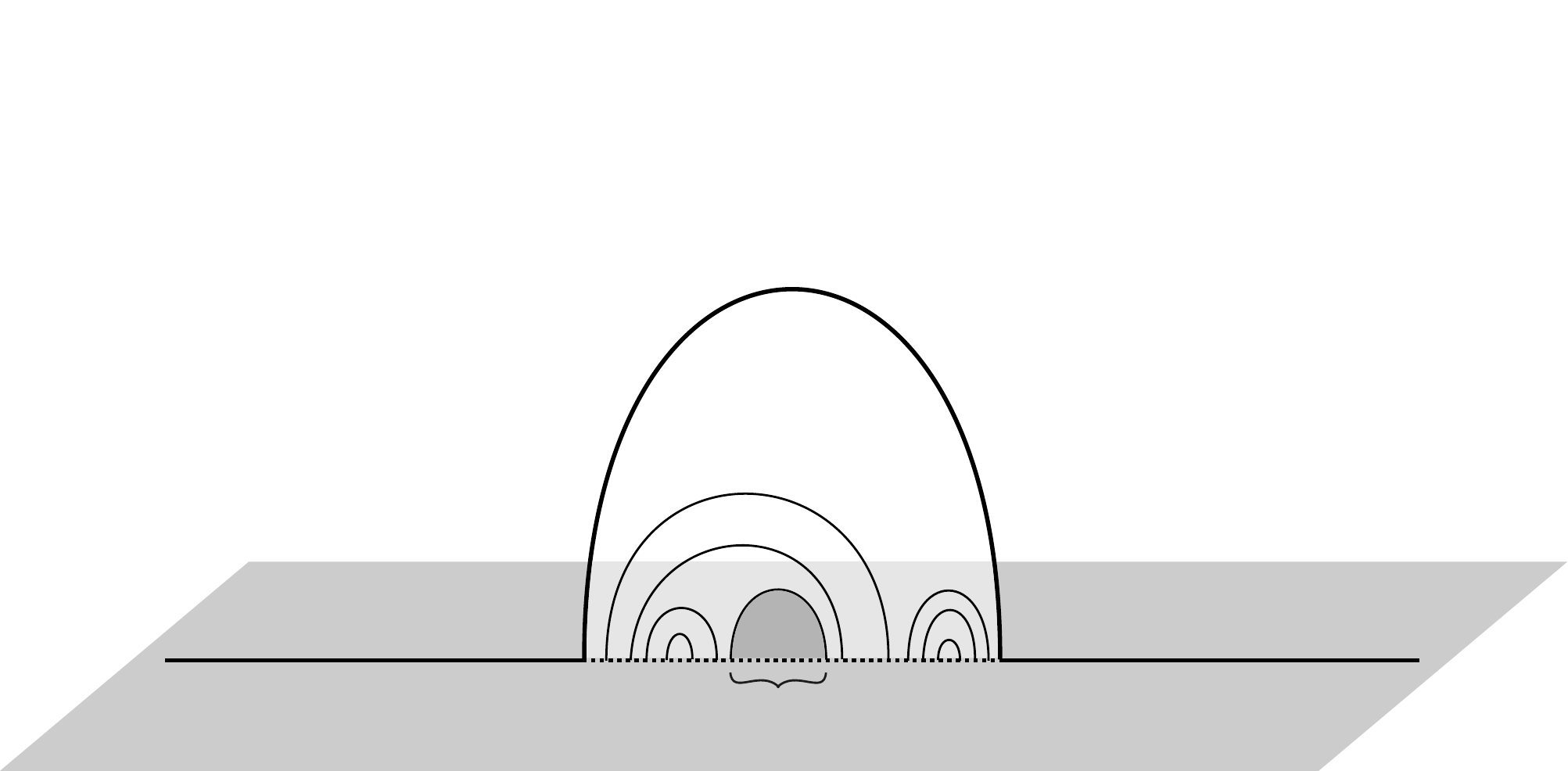} 
\caption{Possible arcs of intersection $R \cap D^k$}
\label{fig:case1lem5.1}
\end{figure}
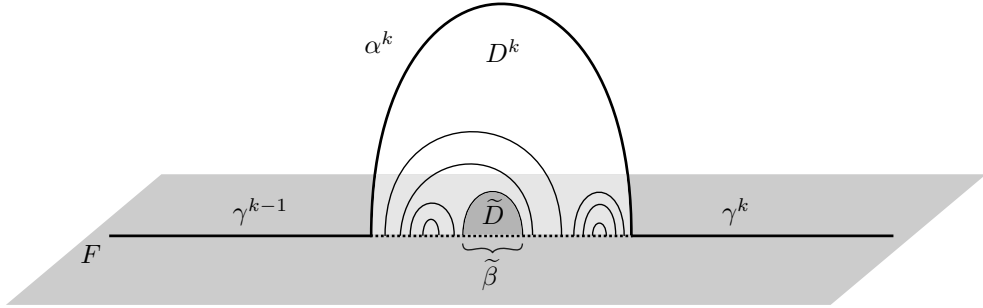
\begin{cor}\label{cor:intersect_boundaryB_or_Gates}
 If $R$ is a red disk above $F$, then $R \cap \left(\partial B' \cup \displaystyle\bigsqcup_{i=1}^{h-1} G^i \right) \neq \emptyset.$
\end{cor}
\begin{proof}
Recall Proposition \ref{prop:flower}: \(\partial B'\) and the gates cut Lab into \(h+1\) components: \(h-1\) colored punctured disks, an annulus, and a twice-punctured disk.
For \(j\geq 2\), \(\gamma^j\) lies entirely in the union of the colored disks and the twice-punctured disk.
(This is clear from Figure \ref{fig:190131GeneralUnlinkDiagram}.)
Since \(\partial R\) must intersect \(\gamma^j\) for some \(j\geq 2\), \(\partial R\) must intersect at least one colored disk or the twice-punctured disk.

On the other hand, since the boundary of a compressing disk cannot be parallel to a puncture, $\partial R$ cannot lie entirely in any of the colored punctured disks. 
\(L\) is an alternating link (since the signs of the rows of twist numbers alternate), and it has a non-split, reduced,  alternating diagram. 
Therefore, by Theorem 1 from \cite{menasco1984closed}, \(L\) is a non-split link.
Since $L$ is nonsplit, $\partial R$ cannot be parallel to $\partial B'$, which implies that $\partial R$ cannot entirely lie in the twice punctured-disk. 
It follows that \(\partial R\) must intersect either a gate or \(\partial B'\).
\end{proof}

\section{All red disks above \(F\) intersect all blue disks below.}

Let \(\mathcal{D}=D^2\cup D^3\cup\cdots \cup D^b\) as before, and let \(\mathcal{R}\) be the subset of red compressing disks above \(F\) which are disjoint from \(\beta'\).
Throughout the rest of this paper, whenever \(\mathcal{R}\) is nonempty, assume that \(R\) is a disk in \(\mathcal{R}\) such that \(|R\cap\mathcal{D}|\leq|R'\cap\mathcal{D}|\) for all \(R'\in\mathcal{R}\), and assume \(R\) is in minimal position with respect to \(\mathcal{D}\) and with respect to the gates.
For each \(i\), \(\beta'\) cuts \(\beta^i\) into multiple subarcs, which we call \textit{lanes}.

The following is Lemma 6.1 from \cite{rodman2018infinite}.
We will not reproduce the proof here.

\begin{lem}\label{lem:two_points_same_lane}
Assume \(\mathcal{R}\) is nonempty.
If, for some \(i\), two points of \(R\cap\beta^i\) lie in the same lane, then they must be endpoints of different components of \(R\cap D^i\).
\end{lem}

By nature of being an element of \(\mathcal{R}\), \(R\) is disjoint from \(\beta'\), and so it follows that \(R\) is also disjoint from \(\partial B'\) (since \(\partial B'\) is the boundary of a regular neighborhood of \(\beta'\)).
By Corollary \ref{cor:intersect_boundaryB_or_Gates}, \(R\) must intersect a gate, and so for some \(i\), \(R\) intersects the \(i\)-colored disk.
Thus the following definition is not vacuous.
For each \(i\), we define each component of intersection of \(\partial R\) with the \(i\)-colored disk to be an \textit{\(i\)-colored track.}
Observe that the colored tracks are pairwise disjoint and have endpoints on the gate of the corresponding color.

For each bridge disk \(D^j\), define the components of \(D^j\cap R\) to be \textit{number \(j\) arcs}, and define the endpoints of each number \(j\) arc to be \textit{number \(j\) points}.\footnote{As a point of notational clarification, we are using \underline{numbers} to label the components of \(\mathcal{D}\) and the arcs \(\mathcal{D}\cap R\) and points \(\mathcal{D}\cap\partial R\) which are contained in \(\mathcal{D}\).
We are using \underline{colors} to label the once-punctured disks of Lab, the gates (which are boundary subarcs of the once-punctured disks), and the tracks (which are components of intersection of the once-punctured disks with \(\partial R\)).
In short, numbers refer to \(\mathcal{D}\), and colors refer to the colored punctured disks.
}

The following lemma is essentially Lemma 6.2 from \cite{rodman2018infinite}, though it is slightly generalized to the case where \(b\geq 3\).
We will not reproduce the proof here since it is the basically same except where we have to consider lanes \(\beta^3,\cdots,\beta^b\) instead of only \(\beta^3\) and \(\beta^4\).
(Also, the proof mentions the disk \(U_2^3\), which in our more general context is \(U_2^{b-1}\).)

\begin{lem}\label{lem:most_numbered_points_lie_in_tracks}
If \(\mathcal{R}\) is nonempty, then at least one endpoint of each number \(2\) arc must lie in a track, and all number \(i\) points must lie in tracks for \(3\leq i\leq b\).
\end{lem}

\begin{figure}
\centering
\labellist \small\hair 2pt
\pinlabel {
    \begin{rotate}{270}
    {\tiny
    1-colored track
    }
    \end{rotate}
    }
    at 13 135
\pinlabel {
    \begin{rotate}{270}
    {\tiny
    2-colored track
    }
    \end{rotate}
    }
    at 58.125 135
\pinlabel {
    \begin{rotate}{270}
    {\tiny
    3-colored track
    }
    \end{rotate}
    }
    at 103.25 135
\pinlabel {
    \begin{rotate}{270}
    {\tiny
    4-colored track
    }
    \end{rotate}
    }
    at 148.38 135
\pinlabel {
    \begin{rotate}{270}
    {\tiny
    5-colored track
    }
    \end{rotate}
    }
    at 193.5 135
\pinlabel {
    \begin{rotate}{270}
    {\tiny
    6-colored track
    }
    \end{rotate}
    }
    at 238.63 135
\pinlabel {
    \begin{rotate}{270}
    {\tiny
    7-colored track
    }
    \end{rotate}
    }
    at 283.75 135
\pinlabel {
    \begin{rotate}{270}
    {\tiny
    8-colored track
    }
    \end{rotate}
    }
    at 328.88 135
\pinlabel {
    \begin{rotate}{270}
    {\tiny
    9-colored track
    }
    \end{rotate}
    }
    at 374 135
\pinlabel {\(\dots\)} at 410 100
%
\pinlabel {
    \begin{rotate}{270}
    {\tiny
    \((h-2)\)-colored track
    }
    \end{rotate}
    }
    at 450 150
\pinlabel {
    \begin{rotate}{270}
    {\tiny
    \((h-1)\)-colored track
    }
    \end{rotate}
    }
    at 493.125 150
%
\pinlabel {\(3\)} at 18 184
\pinlabel {\(3\)} at 63.125 184
\pinlabel {\(4\)} at 108.25 184
\pinlabel {\(4\)} at 153.38 184
\pinlabel {\(5\)} at 198.5 184
\pinlabel {\(5\)} at 243.63 184
\pinlabel {\(6\)} at 288.75 184
\pinlabel {\(6\)} at 333.88 184
\pinlabel {\(7\)} at 379 184
\pinlabel {\(b-1\)} [l] at 445 184
\pinlabel {\(b\)} [l] at 490.125 184
%
\pinlabel {\(2\)} at 18 168
\pinlabel {\(4\)} at 63.125 168
\pinlabel {\(3\)} at 108.25 168
\pinlabel {\(5\)} at 153.38 168
\pinlabel {\(4\)} at 198.5 168
\pinlabel {\(6\)} at 243.63 168
\pinlabel {\(5\)} at 288.75 168
\pinlabel {\(7\)} at 333.88 168
\pinlabel {\(6\)} at 379 168
\pinlabel {\(b\)} [l] at 445 168
\pinlabel {\(b-1\)} [l] at 490.125 168

%
\pinlabel {\(2\)} at 18 30
\pinlabel {\(4\)} at 63.125 30
\pinlabel {\(3\)} at 108.25 30
\pinlabel {\(5\)} at 153.38 30
\pinlabel {\(4\)} at 198.5 30
\pinlabel {\(6\)} at 243.63 30
\pinlabel {\(5\)} at 288.75 30
\pinlabel {\(7\)} at 333.88 30
\pinlabel {\(6\)} at 379 30
\pinlabel {\(b\)} [l] at 445 30
\pinlabel {\(b-1\)} [l] at 490.125 30

%
%
\pinlabel {\(3\)} at 18 15
\pinlabel {\(3\)} at 63.125 15
\pinlabel {\(4\)} at 108.25 15
\pinlabel {\(4\)} at 153.38 15
\pinlabel {\(5\)} at 198.5 15
\pinlabel {\(5\)} at 243.63 15
\pinlabel {\(6\)} at 288.75 15
\pinlabel {\(6\)} at 333.88 15
\pinlabel {\(7\)} at 379 15
\pinlabel {\(b-1\)} [l] at 445 15
\pinlabel {\(b\)} [l] at 490.125 15
\endlabellist
    \includegraphics[width=.7\textwidth]{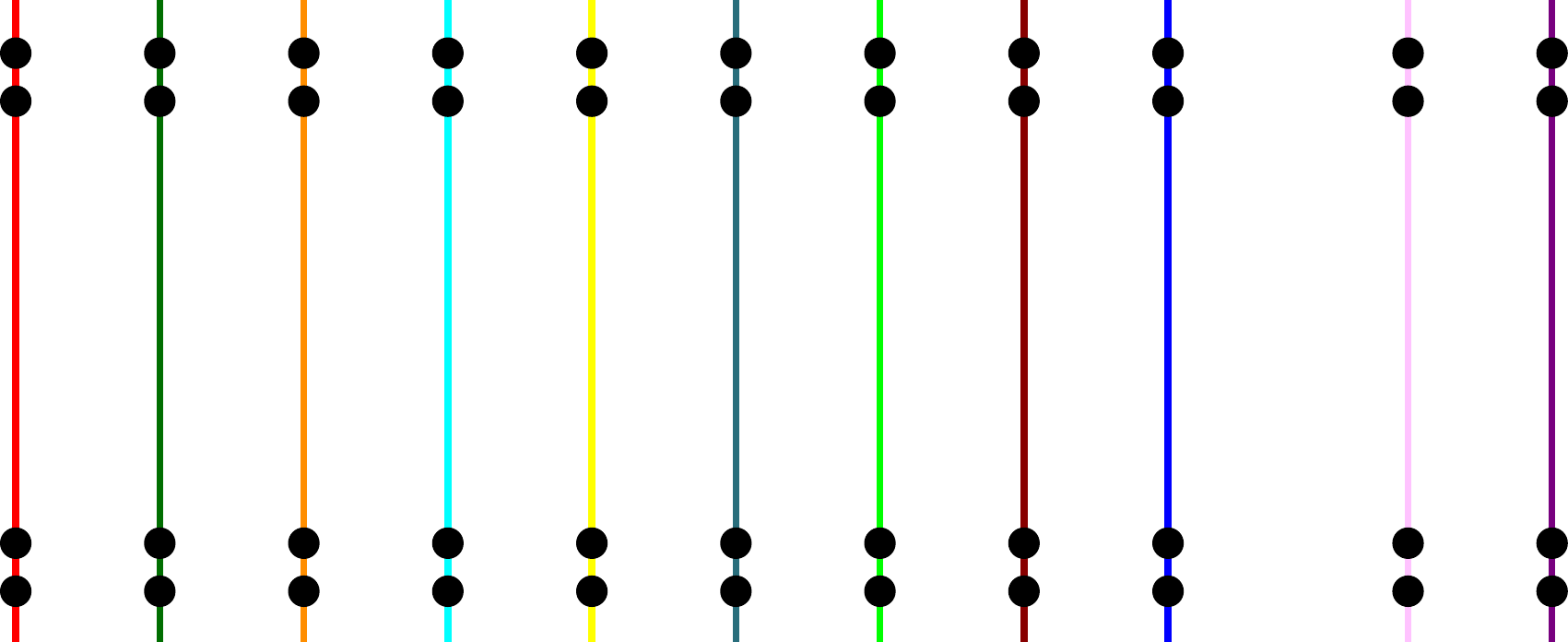}
    \caption{Each colored track has a sequence of numbered points along it, and we can tell what color the track is just from looking at the first and second numbered points.
    Note: there are always \(h-1\) colored tracks.}
    \label{fig:190307TrackSequences2}
\end{figure}
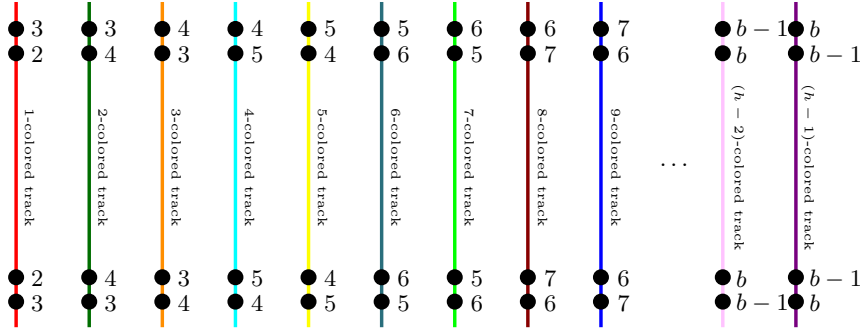

\begin{lem}\label{lem:tracksequences}
The color of each track is uniquely determined by the first and second points in its sequence of numbered points.
\end{lem}

The sequences referred to by Lemma \ref{lem:tracksequences} are depicted in figure \ref{fig:190307TrackSequences2}.

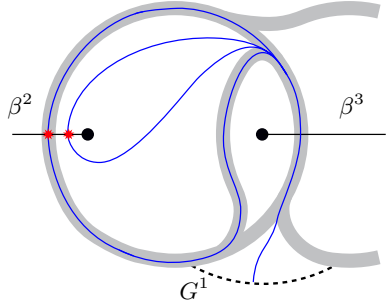
\begin{figure}
\centering
\labellist \small\hair 2pt
\pinlabel {~} at 197 38
\pinlabel {$\beta^3$} at 205 108
\pinlabel {$\beta^2$} at 5 108
\pinlabel {$G^1$} at 110 -2
\endlabellist
\includegraphics[height=.25\textwidth]{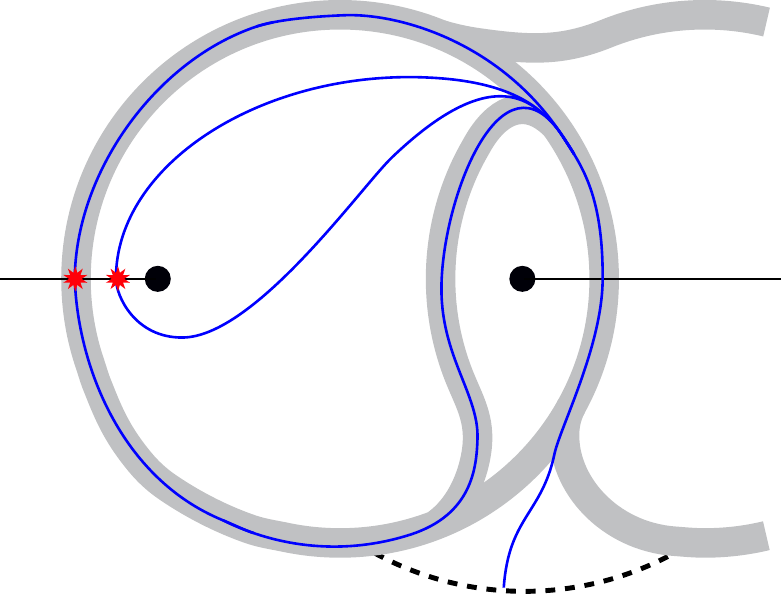}
\caption{The sequence of numbered points along a 1-colored track is \(3,2,\dots,2,3\).}
\label{fig:190228FirstSecond-PenultimateLastPoints_LowerGates_LeftSide}
\end{figure}

\begin{figure}
\centering
\labellist \small\hair 2pt
\pinlabel {$\beta^{j-1}$} at 105 103 
\pinlabel {$\beta^{j}$} at 315 103
\pinlabel {$G^i$} at 225 -5
\endlabellist
\includegraphics[height=.25\textwidth]{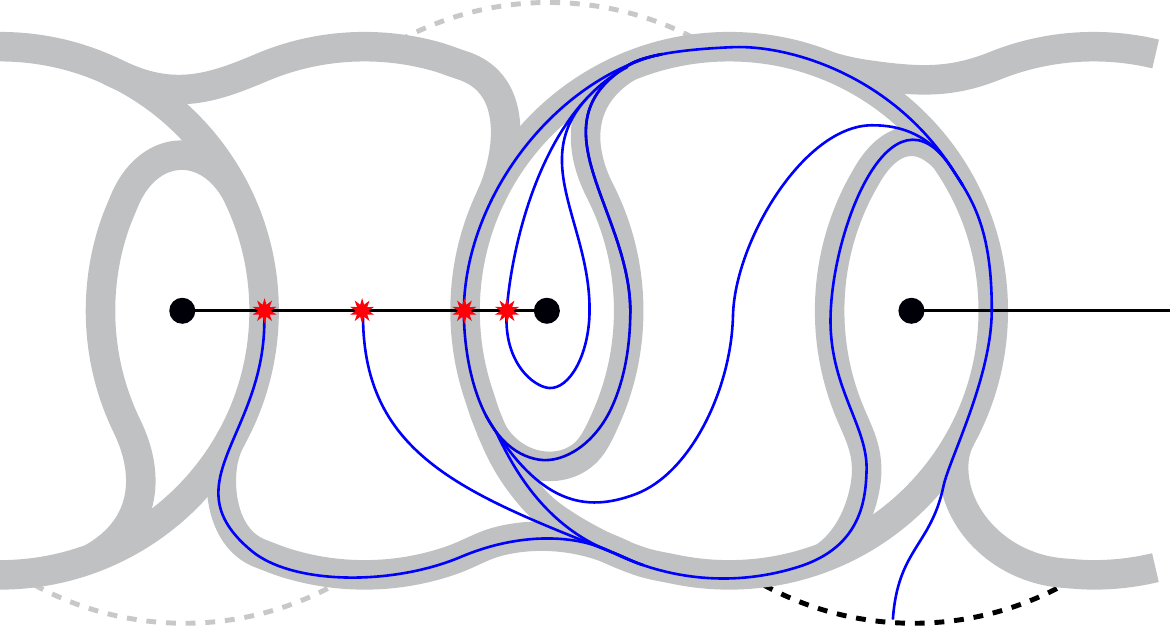}
\caption{For odd \(i\), the sequence of numbered points along an \(i\)-colored track is \(j,j-1,\dots,j-1,j\) (where \(i=2j-5\)).
Here, the case is depicted where $i\geq 3$ (and $j\geq 4$).
}
\label{fig:190228FirstSecond-PenultimateLastPoints_LowerGates}
\end{figure}

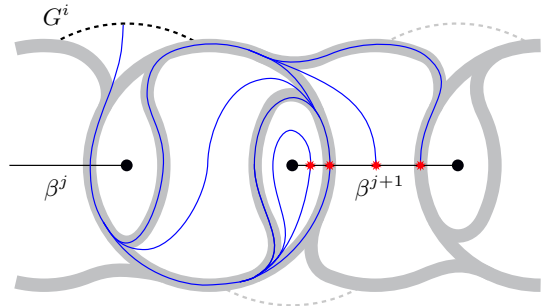
\begin{figure}
\centering
\labellist \small\hair 2pt
\pinlabel {~} at 197 38
\pinlabel {$\beta^{j+1}$} at 235 75
\pinlabel {$\beta^{j}$} at 30 75
\pinlabel {$G^{i}$} at 30 185
\endlabellist
\includegraphics[height=.25\textwidth]{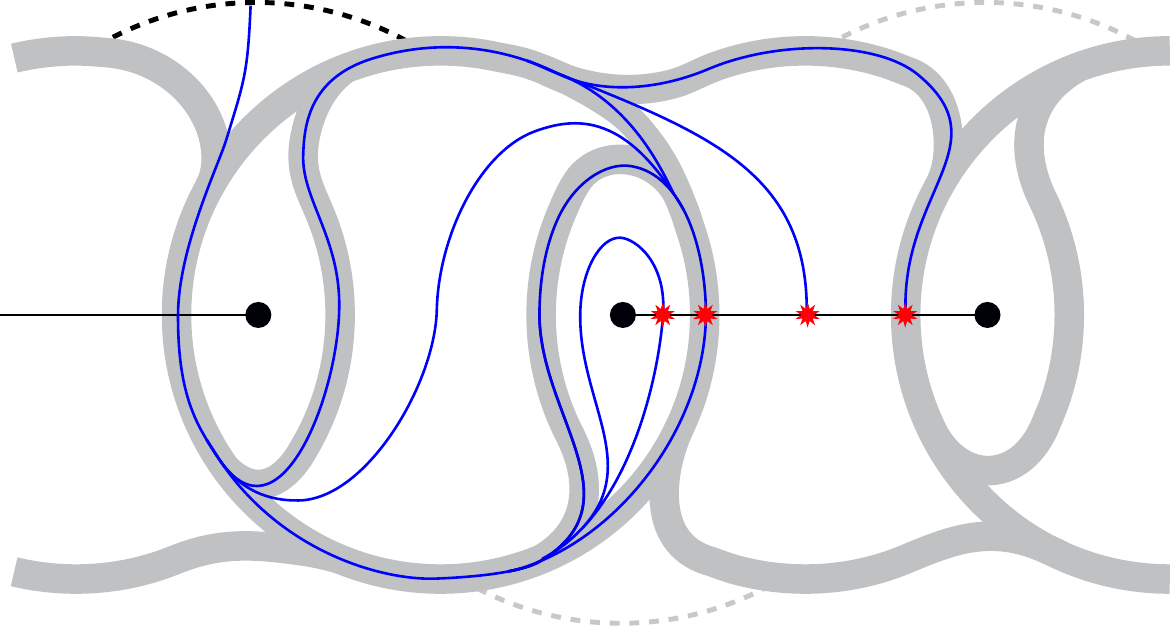}
\caption{For even \(i\), the sequence of numbered points along an \(i\) colored track is \(j,j+1,\dots,j+1,j\) (where \(i=2j-4\)).
This figure depicts the case in which the \(i\) colored gate is not the rightmost upper gate.
(In other words, \(i< h-2\), or equivalently, $j<b-1$.)
}
\label{fig:190228FirstSecond-PenultimateLastPoints_UpperGates}
\end{figure}

\begin{figure}
\centering
\labellist \small\hair 2pt
\pinlabel {$\beta^{b}$} at 265 100
\pinlabel {$\beta^{b-1}$} at 20 100
\pinlabel {$G^{h-2}$} at 30 180
\endlabellist
\includegraphics[height=.25\textwidth]{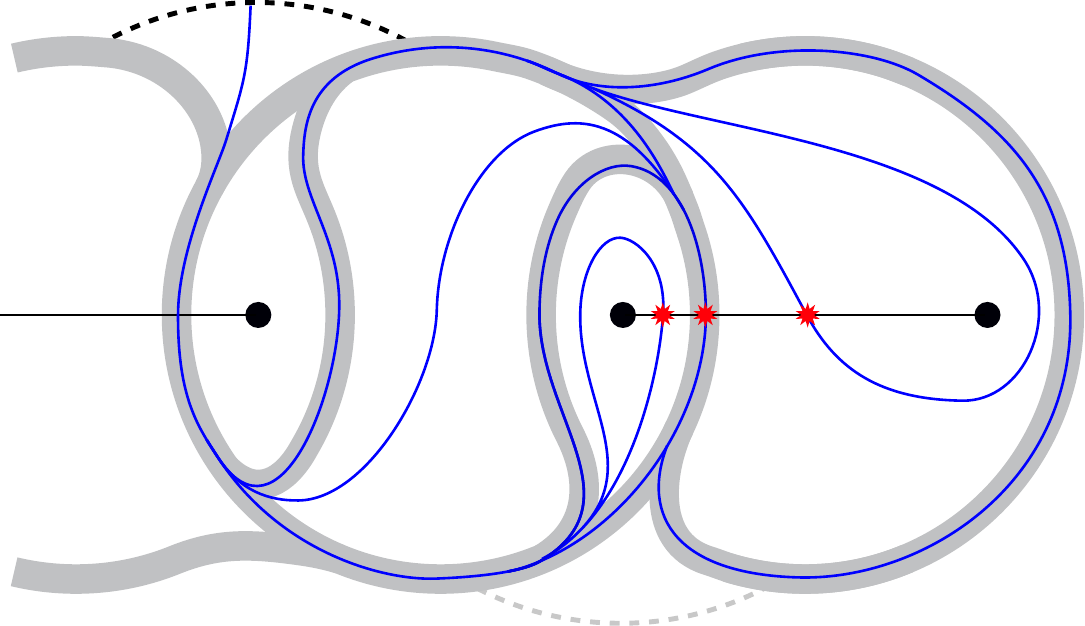}
\caption{Since \(h\) is always even, so is \(h-2\), and here we depict the \(h-2\) colored track.
The sequence of numbered points along such a track is \(b-1,b,\hdots,b,b-1\).}
\label{fig:190228FirstSecond-PenultimateLastPoints_UpperGates_RightSide}
\end{figure}

\begin{proof}
In Figures \ref{fig:190228FirstSecond-PenultimateLastPoints_LowerGates_LeftSide} and \ref{fig:190228FirstSecond-PenultimateLastPoints_LowerGates}, \(i\)-colored tracks are illustrated, where \(i\) is odd.\footnote{Note that when \(i\) is odd, tracks enter from the lower gates, and when \(i\) is even, tracks enter through the upper gates.}
(In Figure \ref{fig:190228FirstSecond-PenultimateLastPoints_LowerGates_LeftSide}, \(i=1\), and in Figure \ref{fig:190228FirstSecond-PenultimateLastPoints_LowerGates}, \(i\geq 3\).)
In each figure, a curve is shown entering the gate, and then every possible choice is depicted at each intersection, demonstrating that in all cases, if the first \(\beta\)-arc intersected is labeled \(\beta^j\), the second must be \(\beta^{j-1}\).
Thus for all \(i\)-colored tracks with \(i\) odd, the sequence of numbered points is \(j, j-1,\dots,j-1,j\).

Similarly, Figures \ref{fig:190228FirstSecond-PenultimateLastPoints_UpperGates} and \ref{fig:190228FirstSecond-PenultimateLastPoints_UpperGates_RightSide} illustrate color \(i\) tracks where \(i\) is even.
(In Figure \ref{fig:190228FirstSecond-PenultimateLastPoints_UpperGates}, \(i\leq h-4\), and in Figure \ref{fig:190228FirstSecond-PenultimateLastPoints_UpperGates_RightSide}, \(i=h-2\).
Recall that the colors range from color \(1\) to color \((h-1)\), and \(h\) is always even.)
The figures show that in all cases, if the first \(\beta\)-arc intersected is \(\beta^j\), then the second must be \(\beta^{j+1}\), so that for all \(i\)-colored tracks with even \(i\), the sequence of numbered points is \(j,j+1,\dots,j+1,j\).

Suppose two tracks are colored \(a\) and \(b\), with \(a\neq b\).
If \(a\) and \(b\) share the same parity, then the tracks are two lower or two upper tracks.
Being different colors, they enter different gates by definition and so they clearly first intersect different \(\beta\)-arcs.
Therefore their sequences of numbered points must differ in the first (and last) number.

On the other hand, suppose \(a\) is odd and \(b\) is even.
Suppose further that each of their sequences' first (and last) points are \(k\) points.
Then the \(a\)-colored track has a second numbered point \(k-1\), and the \(b\)-colored track has a second numbered point \(k+1\), and so their sequences differ in the second (and penultimate) place.
This proves Lemma \ref{lem:tracksequences}.
\end{proof}

We pause to make a few observations from the proof of Lemma \ref{lem:tracksequences}.

\begin{obs}\label{obs:at_least_four_numbered_points_on_each_track}
For each colored track, the corresponding sequence of numbered points contains at least three points.
(The second and penultimate points may coincide.) 
\end{obs}

\begin{obs}\label{obs:2_never_starts_sequence}
The first numbered point in each sequence is never a number 2 point.
\end{obs}

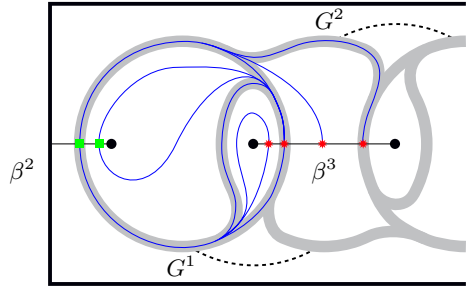
\begin{figure}
\centering
\labellist \small\hair 2pt
\pinlabel {~} at 197 38
\pinlabel {$\beta^{3}$} at 205 85
\pinlabel {$\beta^{2}$} at -19 85
\pinlabel {$G^1$} at 99 15
\pinlabel {$G^2$} at 209 198
\endlabellist
\includegraphics[height=.25\textwidth]{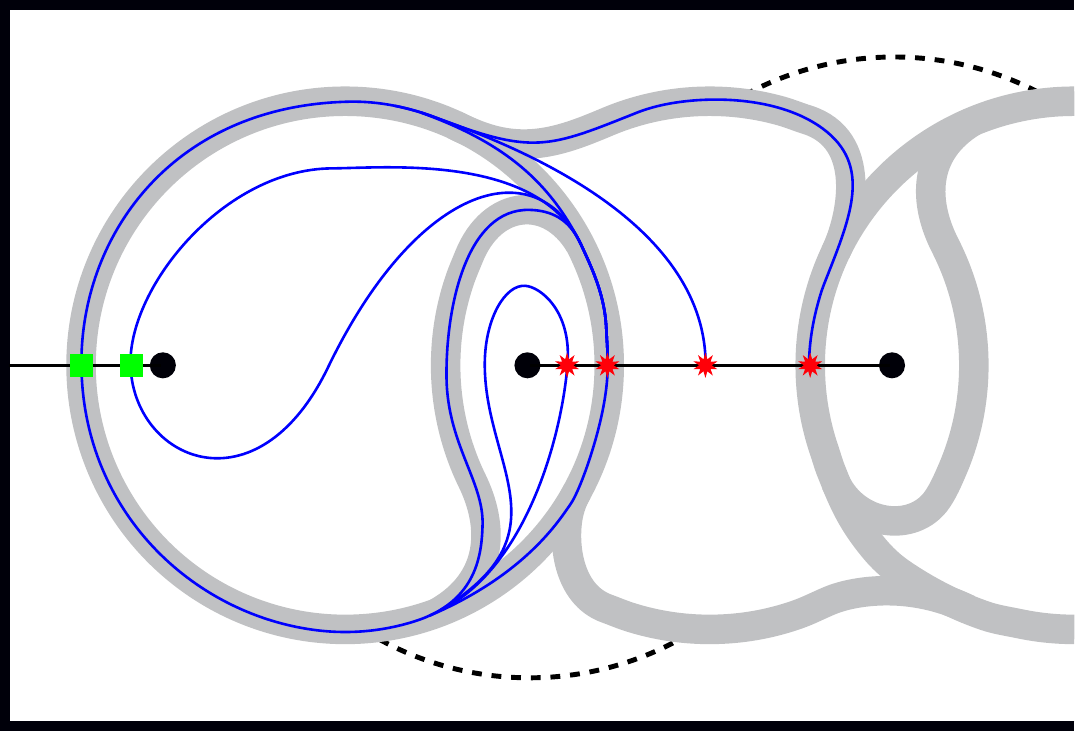}
\caption{After intersecting \(\beta^2\), \(\partial R\) will intersect \(\beta^3\).
}
\label{fig:190308NumberedPointsAlternate_First}
\end{figure}

\begin{figure}
\centering
\labellist \small\hair 2pt
\pinlabel {~} at 197 38
\pinlabel {$\beta^{4}$} at 415 85
\pinlabel {$\beta^{3}$} at 190 85
\pinlabel {$\beta^{2}$} at -20 85
\pinlabel {$G^1$} at 99 15
\pinlabel {$G^2$} at 209 198
\pinlabel {$G^3$} at 309 15
\endlabellist
\includegraphics[height=.25\textwidth]{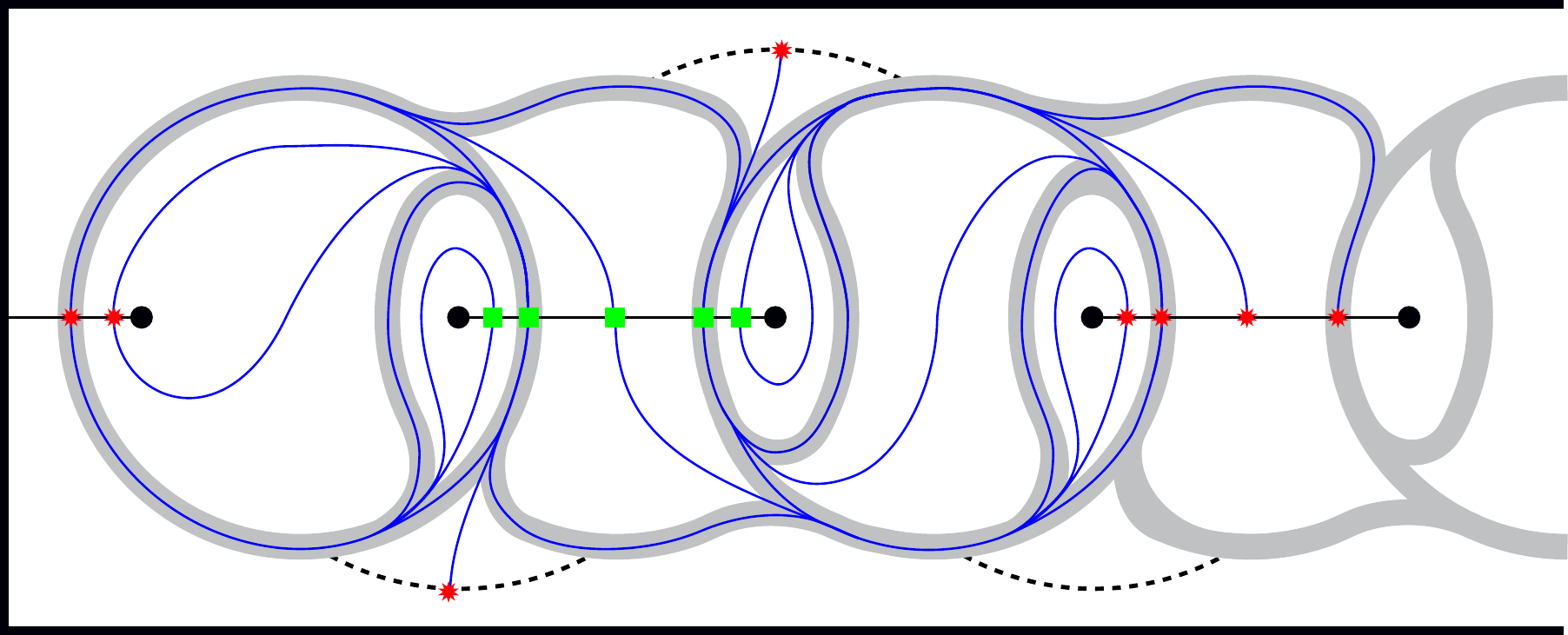}
\caption{After intersecting \(\beta^3\), \(\partial R\) will intersect either a gate or \(\beta^2\) or \(\beta^4\).}
\label{fig:190308NumberedPointsAlternate_Second}
\end{figure}

\begin{figure}
\centering
\labellist \small\hair 2pt
\pinlabel {~} at 197 38
\pinlabel {$\beta^{j}$} at 300 85
\pinlabel {$\beta^{j-1}$} at 110 120
\pinlabel {$\beta^{j+1}$} at 525 85
\endlabellist
\includegraphics[height=.25\textwidth]{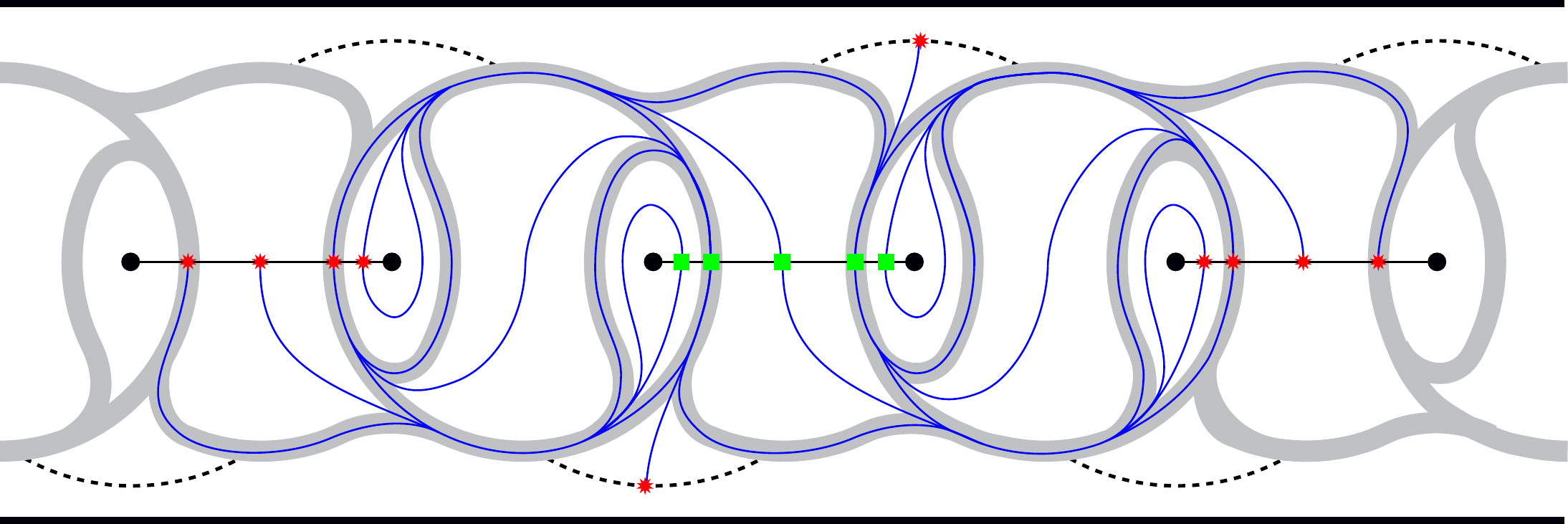}
\caption{After intersecting \(\beta^j\) for \(3<j<b-1\), \(\partial R\) will intersect either a gate or \(\beta^{j-1}\) or \(\beta^{j+1}\).}
\label{fig:190308NumberedPointsAlternate_Middle}
\end{figure}
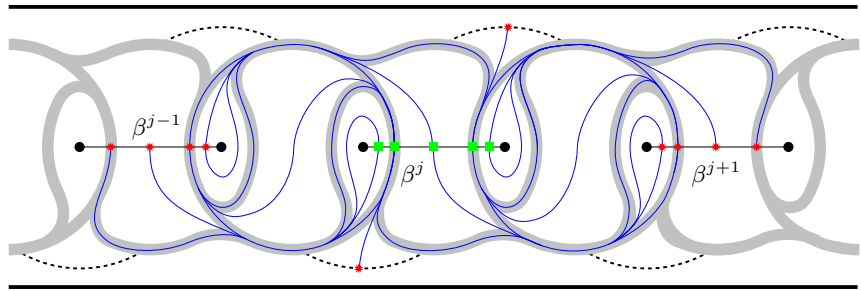

\begin{figure}
\centering
\labellist \small\hair 2pt
\pinlabel {~} at 197 38
\pinlabel {$\beta^{b-1}$} at 302 85
\pinlabel {$\beta^{b-2}$} at 110 120
\pinlabel {$\beta^{b}$} at 549 89
\pinlabel {$G^{h-1}$} at 421 15
\pinlabel {$G^{h-2}$} at 319 195
\pinlabel {$G^{h-3}$} at 311 15
\pinlabel {$G^{h-4}$} at 219 195
\pinlabel {$G^{h-5}$} at 104 15
\endlabellist
\includegraphics[height=.33\textwidth]{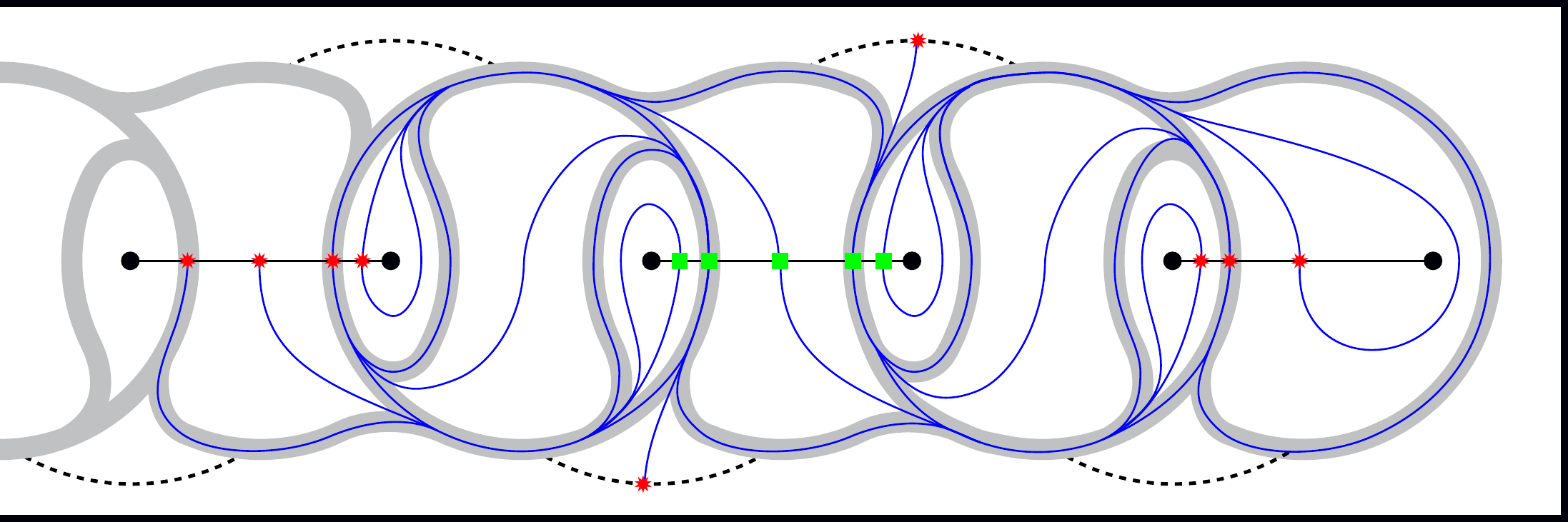}
\caption{After intersecting \(\beta^{b-1}\), \(\partial R\) will intersect either a gate or \(\beta^{b-2}\) or \(\beta^b\).
}
\label{fig:190308NumberedPointsAlternate_Second-to-last}
\end{figure}

\begin{figure}
\centering
\labellist \small\hair 2pt
\pinlabel {~} at 197 38
\pinlabel {$\beta^{b-1}$} at 110 120
\pinlabel {$\beta^{b}$} at 348 92
\pinlabel {$G^{h-1}$} at 315 15
\pinlabel {$G^{h-2}$} at 219 195
\pinlabel {$G^{h-3}$} at 111 15
\endlabellist
\includegraphics[height=.25\textwidth]{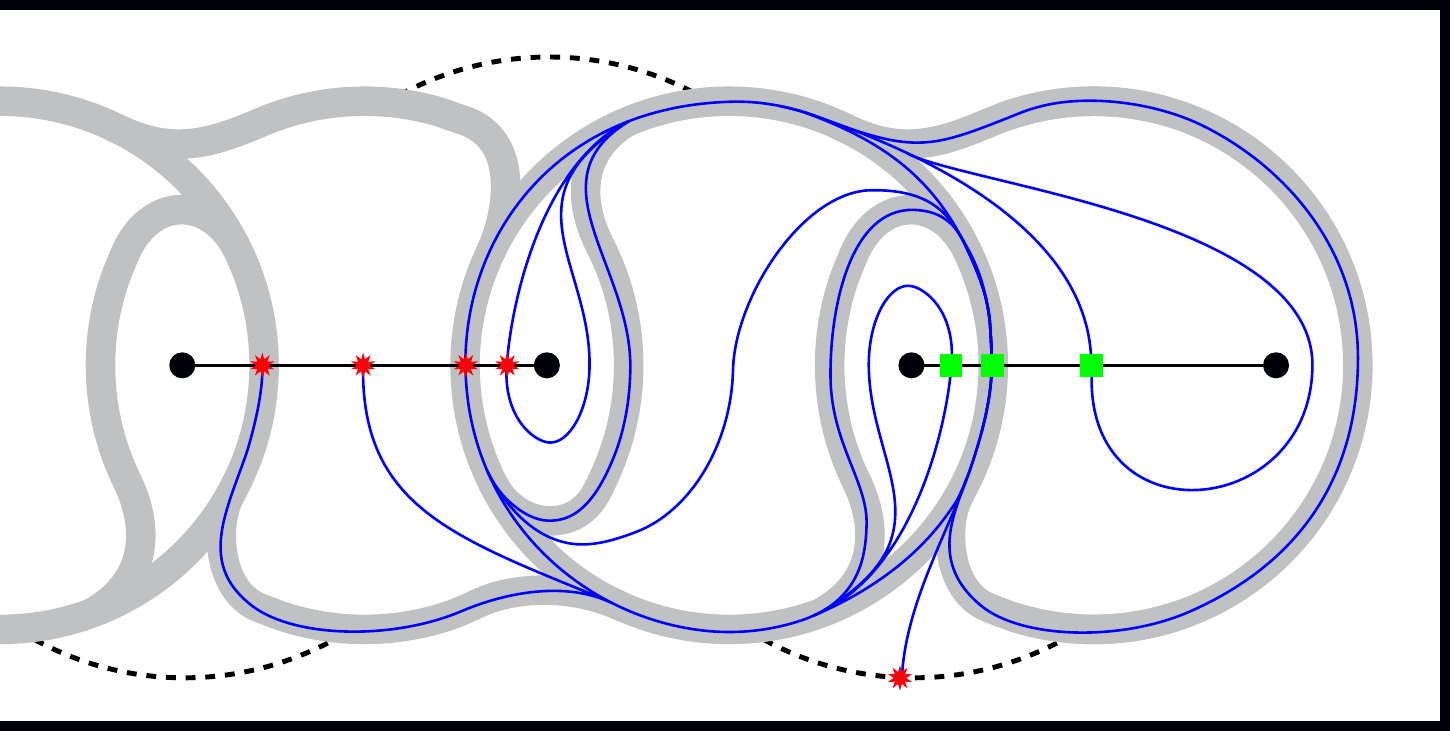}
\caption{After intersecting \(\beta^b\), \(\partial R\) will intersect either the \((h-1)\)-colored gate or \(\beta^{b-1}\).
}
\label{fig:190308NumberedPointsAlternate_Last}
\end{figure}

\begin{lem}\label{lem:alternating}
The numbered points along each track alternate between even and odd numbers.
\end{lem}

\begin{proof}
All possible cases are illustrated in Figures \ref{fig:190308NumberedPointsAlternate_First}, \ref{fig:190308NumberedPointsAlternate_Second}, \ref{fig:190308NumberedPointsAlternate_Middle}, \ref{fig:190308NumberedPointsAlternate_Second-to-last},  and \ref{fig:190308NumberedPointsAlternate_Last}, and in every case, the statement holds true.
\end{proof}

\begin{lem}\label{lem:not_same_track}
Assume \(\mathcal{R}\) is nonempty.
No numbered arc can have endpoints which lie on the same track.
\end{lem}

Lemma \ref{lem:not_same_track} corresponds to Lemma 6.3 from \cite{rodman2018infinite}, where the proof began by showing that along every track, every other numbered point is a number 3 point.
That is not the case in our more general setting, but Lemma \ref{lem:alternating} is a sufficient replacement for that fact, and so we need not reproduce the proof here.

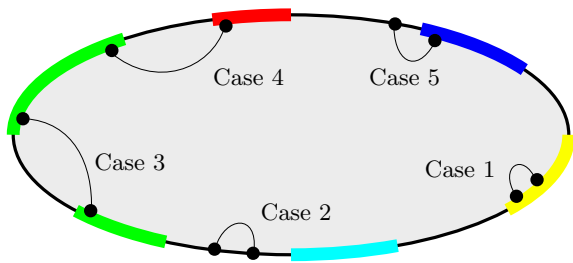
\begin{figure}
\centering
\labellist \small\hair 2pt
\pinlabel {Case 1} [r] at 437 82
\pinlabel {Case 2} [l] at 220 44
\pinlabel {Case 3} [l] at 73 88
\pinlabel {Case 4} [l] at 178 163
\pinlabel {Case 5} at 352 163
\endlabellist
\includegraphics[width=.5\textwidth]{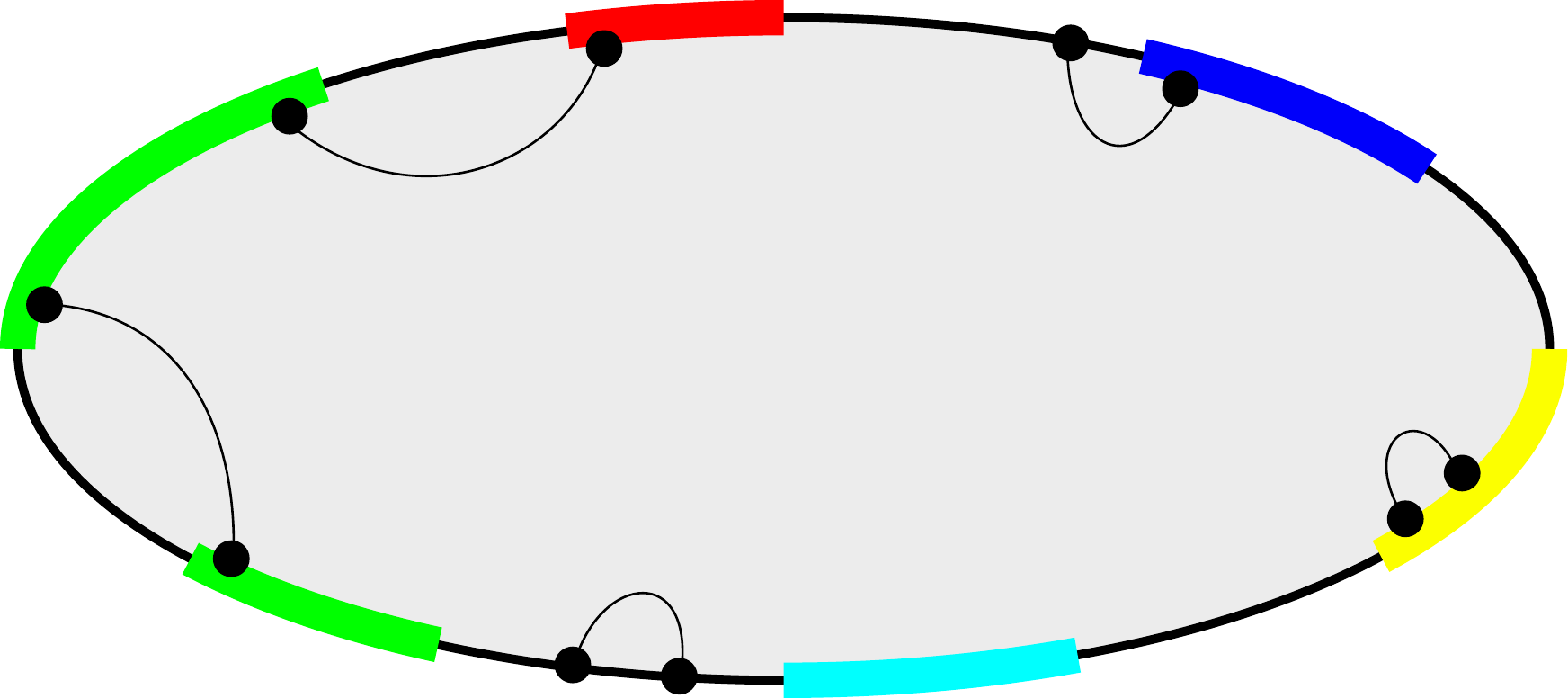}
\caption{Here we show the five possible cases for an outermost arc \(\lambda\) of \(R\cap\mathcal{D}\).
The large gray disk is \(R\), and the thin subarcs properly contained in \(R\) represent possible locations of \(\lambda\).
The bold subarcs around \(\partial R\) represent tracks of various colors.
The points of \(\lambda\cap\partial R\) are numbered points.
}
\label{fig:190418_Intersections_with_R_colored}
\end{figure}

\begin{thm}\label{thm:mathcalR_is_empty}
\(\mathcal{R}\) is the empty set.
\end{thm}

\begin{proof}
Assume for the sake of contradiction that \(R\in\mathcal{R}\).
Let \(\lambda\) be an arc of \(\mathcal{D}\cap R\) which is outermost in \(R\).
The possibilities for the locations of the endpoints of \(\lambda\) are partitioned into the following five cases (illustrated in Figure \ref{fig:190418_Intersections_with_R_colored}).

\begin{description}
	\item[Case 1:] \(\lambda\) connects two points in the same track.
	\item[Case 2:] \(\lambda\) connects two points not on any track.
	\item[Case 3:] \(\lambda\) connects points which are adjacent outermost points of adjacent same-colored tracks.
	\item[Case 4:] \(\lambda\) connects points which are adjacent outermost points of adjacent different-colored tracks.
	\item[Case 5:] \(\lambda\) connects an off-track number 2 point with a point on a track.
\end{description}

Cases 1, 2, 3, and 5 are easily ruled out.
Lemma \ref{lem:not_same_track} rules out Case 1.
In both Cases 2 and 3, the endpoints of \(\lambda\) would lie in the same lane, so 
Lemma \ref{lem:two_points_same_lane} rules these cases out.

In Case 5, exactly one endpoint of \(\lambda\) lies on a track.
Call this track \(T\), and let \(P_{T}\) and \(P'\) denote the endpoints of \(\lambda\) on \(T\) and off of \(T\), respectively.
Since all points off of tracks are number 2 points, \(P'\) is a number 2 point, which implies \(\lambda\) is a number 2 arc, and thus \(P_T\) is a number 2 point.
Since \(\lambda\) is outermost, \(P_T\) must be the first numbered point in the sequence along \(T\).
But recall Observation \ref{obs:2_never_starts_sequence}: The first numbered point in each sequence of numbered points is never a number 2 point.
This contradiction rules out Case 5.

\begin{figure}
\centering
\labellist \small\hair 2pt
\pinlabel {\(\zeta\)} at 238 111
\pinlabel {\(\lambda\)} at 217 65
\pinlabel {\(n\)} at 250 19
\pinlabel {\(n\pm 1\)} at 318 38
\pinlabel {\(n\)} at 144 31
\pinlabel {\(n\pm 1\)} at 90 58
\endlabellist
\includegraphics[width=.55\textwidth]{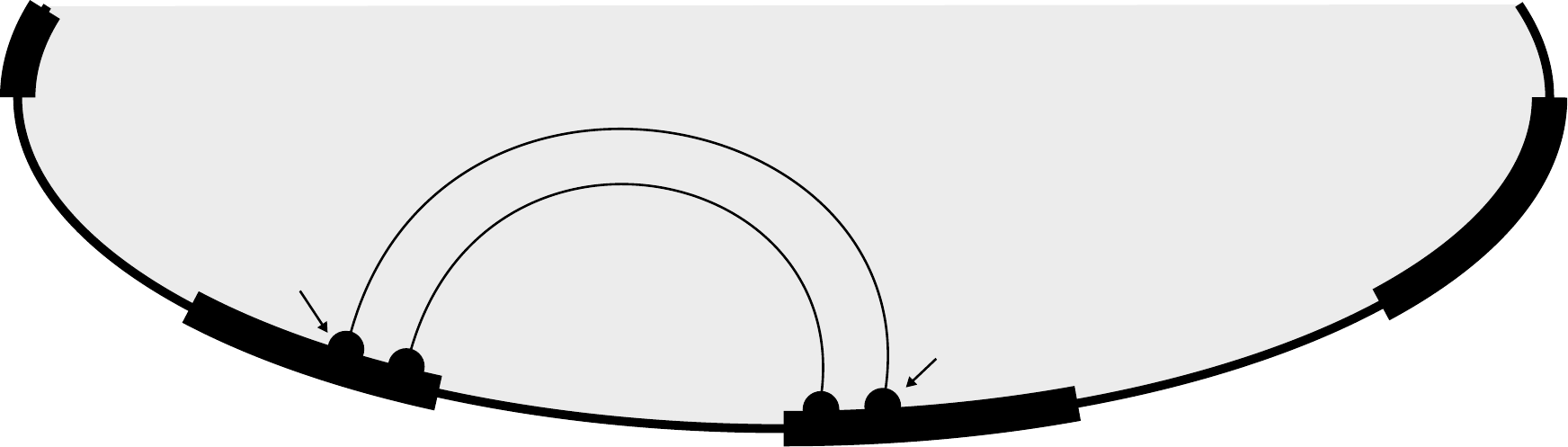}
\caption{The gray disk is \(R\), and the bold subarcs around \(\partial R\) are variously colored tracks. 
This figure illustrates that Case 4 would imply that the tracks on which the endpoints of \(\lambda\) lie are simultaneously different colors and the same color, a contradiction.}
\label{fig:190424_Outermost_SecondOutermost}
\end{figure}
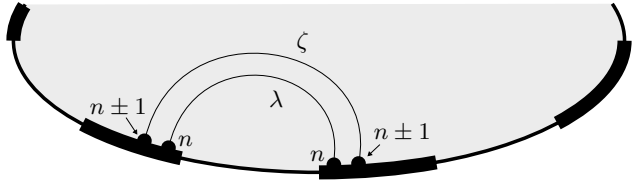

The only remaining case is Case 4.
We will need to establish the existence of second-outermost arcs of intersection of \(\mathcal{D}\) and \(R\).
By Observation \ref{obs:at_least_four_numbered_points_on_each_track}, given the sequence of numbered points along any track, non-outermost points exist in the sequence.
Let \(Q\) be a number \(j\) point which is not the first or last point in a sequence of numbered points along a track.
\(Q\) is by definition the endpoint of a number \(j\) arc \(\xi\). 
By Lemma \ref{lem:alternating}, the numbered points adjacent to \(Q\) in the sequence are not number \(j\) points, and so neither of them can be an endpoint of \(\xi\), which shows that \(\xi\) is not outermost.
This establishes the existence of non-outermost numbered arcs, from which it follows that second-outermost numbered arcs exist.

Let \(\zeta\) be an arc of \(\mathcal{D}\cap R\) which is a second-outermost arc in \(R\).
By definition, \(\zeta\) cuts off a disk of \(R\) containing a single arc of \(\mathcal{D}\cap R\), which is therefore an outermost arc.
Call this outermost arc \(\lambda\).

The arc \(\lambda\) must fit into Case 4, as we have determined all the other cases are impossible.
That is, the endpoints of \(\lambda\) must be points which are adjacent outermost points of adjacent different-colored tracks.
Then the endpoints of \(\zeta\) must be the second-outermost points of these two different-colored tracks.
Since the two endpoints of \(\lambda\) must have the same number as each other, and the two endpoints of \(\zeta\) must have the same number as each other, Lemma \ref{lem:tracksequences} implies that these two different colored tracks are the same color, a contradiction.
Therefore, all five cases lead to contradictions, completing the proof that \(R\) cannot exist.
\end{proof}

\begin{thm}\label{thm:mainmaintheorem}
The canonical bridge sphere for \(L\) is critical.
\end{thm}

\begin{proof}
Observe that $B$ and $B'$ give a pair of disjoint blue disks on opposite sides of $F$. Let $B''$ and $B'''$ be the frontier of a regular neighborhood of \(D''\) in $M_-$ and the frontier of a regular neighborhood of \(D^b\) in $M_+$ respectively. It is easy to see that $B''$ and $B'''$ give a pair of disjoint red disks on opposite sides of $F$. Thus, condition (1) in the definition of a critical surface is satisfied. To see that condition (2) is also satisfied, note that Theorem \ref{thm:mathcalR_is_empty} implies that any red disk above \(F\) intersects the blue disk below \(F\). Furthermore, by rotational symmetry, we can slightly alter the arguments presented above to show that any red disk below \(F\) intersects the blue disk above \(F\). Hence, we may conclude that the canonical bridge sphere for \(L\) is critical. 
\end{proof}

For each twist region of our link, there are various ways to appropriately select the number of half twists so that the resulting link has $k$ components where $k \in \mathbb{N},$ and $1 \leq k \leq b.$ 
(For example, we could accomplish this by assigning an odd number of half twists to exactly \(b-k\) of the twist regions along the top row, and assigning an even number of half twists to all the other twist regions.)
In particular, we obtain the following corollaries (the second of which is stronger):

\begin{cor}
There is an infinite family of nontrivial knots with critical bridge spheres.
\end{cor}

\begin{cor}
Given any bridge number \(b\geq 3\), and any integer \(k\) with \(1\leq k\leq b\), there exists an infinite family of \(k\)-component links in \(b\) bridge position with critical bridge spheres.
\end{cor}

\bibliographystyle{plain}
\bibliography{main}
\end{document}